\documentclass[a4paper,11pt]{article}
\usepackage{newclude}
\usepackage[nogytt]{mfpaperstuff2}
\usepackage{multicol}
\usepackage{arydshln}

\begin{document}
\hyphenation{mon-oid}
\hyphenation{mon-oids}
\hyphenation{twent-ieth}
\newcommand\starr{(\textasteriskcentered)\xspace}
\newcommand\codim{\operatorname{codim}}
\newcommand\pwi[2]{\mathrm{prj}(#1,#2)}
\newcommand\ud[1]{\widebar{#1}}
\newcommand\xli{X-line\sxpace}
\newcommand\sx[2]{(#1|#2)}
\newcommand\nrg{normalising reflection group\xspace}
\newcommand\ol\bar
\newcommand\cor{completely reducible\xspace}
\newcommand\kerm[1]{\operatorname{Ker}(#1)}
\newcommand\imm[1]{\operatorname{Im}(#1)}
\newcommand\as[2]{X_{#1}^{(#2)}}
\newcommand\bw[1]{Y_{#1}}
\newcommand\ci[1]{Z^{(#1)}}
\newcommand\poj[2]{\mathrm{prj}(#1,#2)}
\newcommand\eend{\operatorname{End}}
\renewcommand\mp{maximal projection\xspace}
\newcommand\mps{maximal projections\xspace}
\newcommand\fmx[4]{\begin{pmatrix}#1&#2\\#3&#4\end{pmatrix}}
\newcommand\amx[6]{\left(\!\!\begin{array}{cc|c}#1&#2&#3\\#4&#5&#6\end{array}\!\!\right)}
\newcommand\gl[2]{\operatorname{GL}_{#1}(#2)}
\newcommand\gll[1]{\operatorname{GL}(#1)}
\newcommand\im{\operatorname{im}}
\renewcommand\tr[1]{\operatorname{Tr}(#1)}
\newcommand\co{\ensuremath{C_0}}
\newcommand\spn[2]{\genfrac{\langle}{\rangle}{0pt}{}{#1}{#2}}
\newcommand\spa[3]{\left\langle\begin{smallmatrix}#1\\#2\\#3\end{smallmatrix}\right\rangle}
\newcommand\ik[1]{\operatorname{IK}(#1)}
\newcommand\cmp{complete\xspace}
\newcommand\cmpness{completeness\xspace}
\newcommand\stm{structure matrix\xspace}
\newcommand\Srm{Semi\-reflection mon\-oid\xspace}
\newcommand\srm{semi\-reflection mon\-oid\xspace}
\newcommand\Srms{\Srm{}s\xspace}
\newcommand\srms{\srm{}s\xspace}
\newcommand\prm{projection monoid\xspace}
\newcommand\prms{\prm{}s\xspace}
\newcommand\Prm{Projection monoid\xspace}
\newcommand\Prms{\Prm{}s\xspace}
\newcommand\mna[1]{\cala_{#1}}
\newcommand\mnc[1]{\calc_{#1}}
\newcommand\mnd[1]{\cald_{#1}}
\newcommand\mnap[1]{\mna{#1}^+}
\newcommand\mnb[1]{\calb_{#1}}
\newcommand\ze\zeta
\newcommand\frf[1]{\check{#1}}
\title{Monoids generated by projections}
\runninghead{Monoids generated by projections}

\toptitle

\begin{abstract}
We define and explore \srms on a finite-dimensional vector space. These are monoids generated by \emph{semireflections}: linear maps fixing a subspace of codimension $1$. We mostly focus on the case of \prms (where the generating semireflections are non-invertible).  After exploring some general theory, we give some important examples, and give classification results for \prms on $\bbc^2$ and $\bbr^3$. We then briefly introduce affine \prms.
\end{abstract}

\tableofcontents

\section{Introduction}

The theory of finite reflection groups (both real and complex) is one of the cornerstones of twentieth-century mathematics, with applications in numerous areas, and continues to inspire a great deal of research. We begin by summarising the essentials.

Throughout this paper we let $V$ be a finite-dimensional vector space over a field $\bbf$. We define a \emph{reflection} on $V$ to be an invertible linear map $r:V\to V$ whose fixed-point space has codimension $1$ (or equivalently, for which the space $\lset{(r-1)v}{v\in V}$ has dimension $1$). A \emph{reflection group} on $V$ is a subgroup of $\gll V$ generated by reflections. We are interested primarily in \emph{finite} reflection groups. Obviously in a finite reflection group each reflection $r$ has finite order, so has eigenvalues $1,\dots,1,\ze$ for a root of unity $\ze\in\bbf$. Moreover, if $\nchar(\bbf)=0$ then $\ze\neq1$, so that $r$ is diagonalisable. (What we have called a reflection is often called a \emph{pseudoreflection}, with the term `reflection' being reserved for the case where $\ze=-1$.)

A finite reflection group $G$ on $V$ is \emph{irreducible} if there is no subspace $0<W<V$ which is preserved by $G$. It is \emph{\cor} if for every $W_1\ls V$ which is preserved by $G$ there is another subspace $W_2$ which is preserved by $G$ such that $V=W_1\oplus W_2$. In this case $G$ decomposes as a direct product $G_1\times G_2$, where $G_i$ is the pointwise stabiliser of $W_{3-i}$, and acts as a \cor reflection group on $W_i$. Iterating this decomposition shows that for any \cor finite reflection group $G$ on $V$ there is a direct sum decomposition $V=W_1\oplus\dots\oplus W_r$ and a corresponding direct product decomposition $G=G_1\times\dots\times G_r$, where $G_i$ acts as an irreducible reflection group on $V_i$ and fixes $V_j$ pointwise for each $j\neq i$.

It is essentially a consequence of Maschke's Theorem that if $\nchar(\bbf)=0$ then every finite reflection group is \cor, so to classify finite reflection groups it suffices to classify irreducibles. Over $\bbr$ this was done by Coxeter \cite{cox1}, who showed \cite{cox2} that the irreducible finite real reflection groups are precisely the finite Coxeter groups. Over $\bbc$ the irreducible finite reflection groups were classified by Shephard and Todd \cite{st}; these consist of the infinite family of groups $G(m,p,n)$, together with 34 exceptional groups $G_4,\dots,G_{37}$. Finite reflection groups in prime characteristic have been considered in a series of papers by Zaleskii and Serezhkin \cite{zs1,zs2,zs3}.

Now we introduce the generalisations of reflection groups that we study in this paper. Suppose $V$ is a finite-dimensional vector space over $\bbf$. We define a \emph{semireflection} to be a (not necessarily invertible) linear map $r:V\to V$ whose fixed-point space has codimension $1$. So a semireflection is either a reflection or a projection onto a subspace of dimension $n-1$. For the rest of this paper, we use the term `projection' to mean a projection onto a subspace of codimension $1$; or in other words, an idempotent linear map of nullity~$1$.

We define a \emph{\srm} on $V$ to be a submonoid of $\eend_\bbf(V)$ generated by semireflections. A \srm is a \emph{\prm} if it is generated by projections. We will mainly restrict attention to \prms in this paper; we are particularly interested in finding \emph{finite} \prms. We will concentrate on irreducibility and classification, leaving the study of the algebraic structure of these monoids to future work.

In the next section we will explore some basic theory of \srms and \prms, particularly relating to irreducibility and complete reducibility.  In \cref{exsec} we introduce two important examples and give come of their properties. In \cref{classsec,r3classsec} we classify irreducible \prms on $\bbc^2$ and on $\bbr^3$, and in the final section we introduce affine \prms.

\section{\Srms}

\subsection{Some more definitions}\label{somemore}

Here we give some more basic definitions relating to \srms. First suppose $M,N$ are \srms on vector spaces $V,W$ respectively. Then we say that $M$ and $N$ are \emph{equivalent} if there is a bijective linear map $f:V\to W$ such that $N=\lset{fmf^{-1}}{m\in M}$.

If $r:V\to V$ is a semireflection, then so is the dual map $r^\ast:V^\ast\to V^\ast$. So if $M$ is a \srm on $V$, then there is a dual \srm $M^\ast$ on $V^\ast$. Unlike in the case of finite reflection groups, a finite \srm need not be equivalent (or even isomorphic an an abstract monoid) to its dual.

Now we show how to decompose a \srm into a `reflection part' and a `projection part'. Suppose $M$ is a \srm. Then we define $M_0$ to be the submonoid generated by all the projections in $M$, and $M_1$ the submonoid generated by the reflections in $M$. Then $M_1$ is simply the set of invertible elements of $M$ and is a reflection group. Moreover, $M_0\cap M_1=\{1\}$, $M_0M_1=M$, and $M_1$ normalises $M_0$. (In many cases $M_0$ comprises all the non-invertible elements of~$M$, but this is not true in general.)

Given a \prm $M_0$, we define its \emph{\nrg} to be the group generated by all the reflections that normalise $M_0$. Now to construct a finite \srm, we can proceed in stages:
\begin{itemize}
\item
construct a finite \prm $M_0$;
\item
let $G$ be the \nrg of $M_0$;
\item
let $M_1$ be a finite subgroup of $G$, and let $M$ be the \srm $M_0M_1$.
\end{itemize}

In this paper we will mainly be concerned with the first step, though for many of the \prms that we construct we will find the \nrg.

We remark that one can consider instead \emph{projection semigroups} instead of \prms, with no practical difference; a projection semigroup is the same as the corresponding \prm but with the identity element removed. The choice is largely a matter of taste; our choice to work with monoids simplifies some statements (e.g.\ the expression $M_0M_1$ above would need to be $M_0M_1\cup M_1$ when working with semigroups, and references to direct products below would need modification) while complicating others (in \cref{exsec} below we repeatedly make reference to the non-invertible elements of a \prm, which are simply the elements of the corresponding projection semigroup).

\subsection{Irreducibility and complete reducibility}\label{irrsec}

Suppose $M$ is a \prm on $V$. We say that $M$ is \emph{reducible} if there is a subspace $0<W<V$ which is mapped to itself by all elements of $M$. We call such a subspace \emph{invariant} for $M$. We say that $M$ is \emph{irreducible} if it is not reducible.

We begin with some basic theory which parallels the corresponding theory for reflection groups. We leave the proof of the following simple \lcnamecref{invlem} to the reader.

\begin{lemma}\label{invlem}
Suppose $M$ is a \prm on $V$, and $0\ls W\ls V$. Then $W$ is invariant for $M$ \iff for every projection $p\in M$ either $\ker(p)\ls W$ or $W\ls\im(p)$.
\end{lemma}

This lemma allows us to further examine the structure of a reducible \prm. Suppose $M$ is a \prm and $W$ is an invariant subspace for $M$. For each projection $p\in M$, we consider the two possibilities in \cref{invlem}. If $\ker(p)\ls W$, then $p$ yields a projection on $W$ and the identity on $V/W$; on the other hand, if $W\ls\im(p)$, then $p$ acts as the identity on $W$ and yields a projection on $V/W$.

So we naturally obtain a \prm $M_W$ on $W$, generated by the restrictions to $W$ of the projections in $M$, and a \prm $M_{V/W}$ on $V/W$, generated by the projections induced by the projections in $M$ that fix $W$. The direct product $M_W\times M_{V/W}$ is naturally a quotient of $M$.

Now say that $M$ is a \emph{\cor} if for every invariant subspace $W$ there is another invariant subspace $X$ with $V=W\oplus X$, or equivalently, if the sum of the minimal invariant subspaces for $M$ is $V$.

In the case where $V$ is the direct sum $W_1\oplus W_2$ of two invariant subspaces, we can refine the above discussion. 
By \cref{invlem} we can partition the set of projections in $M$ as $P_1\sqcup P_2$, where $\ker(p)\ls W_i$ and $W_{3-i}\ls\im(p)$ for each $p\in P_i$. Then each projection in $P_1$ commutes with each projection in $P_2$. If we let $M_i$ be the monoid generated by $P_i$ for $i=1,2$, then every $m\in M_i$ fixes $W_{3-i}$ pointwise. So $M_1\cap M_2=\{1\}$ and $M\cong M_1\times M_2$. Each $p\in P_i$ acts as a projection on $W_i$, so we can view $M_i$ as a \prm on $W_i$. If $M$ is \cor, then so is $M_i$ (as a \prm on $W_i$).

Iterating this, we find that if $M$ is a \cor \prm on $V$, then there is a direct sum decomposition $V=V_1\oplus\dots\oplus V_r$ and a direct product decomposition $M=M_1\times\dots\times M_r$, where each $M_i$ is an irreducible \prm on $V_i$ (and acts as the identity on $V_j$ for each $j\neq i$). So to classify \cor \prms, it suffices to classify irreducible ones.

\begin{eg}
In this example we describe all reducible \prms in dimension $2$ (for any $\bbf$). As a consequence, we observe that (in contrast to the situation with reflection groups) finite \prms need not be \cor in characteristic $0$.

Suppose $\dim(V)=2$ and let $L$ be any line in $V$. We consider the \prms on $V$ having $L$ as an invariant subspace. By \cref{invlem} any projection $p$ in such a \prm must have $L$ as an eigenspace (that is, $L$ must be either the kernel or the image of $p$).

Let $P$ be \emph{any} finite set of projections all having $L$ as an eigenspace, and let $M$ be the monoid generated by $P$. Then we claim that $M$ is finite. To see this, suppose $p,q\in P$, and observe that:
\begin{itemize}
\item
if $L=\ker(p)=\ker(q)$, then $pq=q$;
\item
if $L=\im(p)=\im(q)$, then $pq=p$;
\item
if $L=\im(p)=\ker(q)$, then $pq=0$. 
\end{itemize}
As a consequence, any expression in the elements of $P$ reduces to an expression of length at most $2$ or to $0$, so $M$ is finite.

$M$ is reducible (because $L$ is invariant for $M$), but generally not \cor.
\end{eg}

\subsection{Images and kernels}

In this subsection we describe necessary and sufficient conditions for complete reducibility of a \prm $M$; here the theory diverges from the theory of finite reflection groups.

First we set out some definitions which will be used throughout the rest of the paper. Suppose $M$ is a \prm on $V$. Given subspaces $K,L<V$ with $\dim(K)=\codim(L)=1$ and with $K\nls L$, we write $\pwi KL$ for the projection with kernel $K$ and image $L$. We Say that a line $K$ is a \emph{kernel} of $M$ if $\pwi KL\in M$ for some $L$, and we say that a hyperplane $L$ is an \emph{image} of $M$ if $\pwi KL\in M$ for some $K$. We write $\kerm M$ for the set of kernels of $M$, and $\imm M$ for the set of images of $M$.

We say that $M$ is \emph{\cmp} if the following condition holds: whenever $K\in\kerm M$ and $L\in\imm M$ with $K\nls L$, the projection $\pwi KL$ lies in $M$.

\medskip
Now we can give the main theorem of this section. We define the \emph{trivial} \prm on $V$ to be the monoid containing just the identity element; clearly this is irreducible \iff $\dim(V)=1$.

\begin{thm}\label{irrconds}
Suppose $M$ is a non-trivial finite irreducible \prm on $V$. Then:
\begin{itemize}
\item
the sum of the kernels of $M$ is $V$;
\item
the intersection of the images of $M$ is $0$; and
\item
$M$ is \cmp.
\end{itemize}
\end{thm}

\begin{pf}
Let $W$ be the sum of the kernels of $M$. If $v\in W$ and $p$ is a projection in $M$, then $p(v)-v$ lies in the kernel of $p$ and therefore lies in $W$, so $p(v)\in W$. So $W$ is invariant for $M$. Because $M$ is non-trivial, $W$ must be non-zero, so (because $M$ is irreducible) $W$ must equal $V$. The second statement follows in a similar way (or by duality).

Now we prove that $M$ is \cmp. Suppose $K$ is a kernel of $M$ and $L$ an image of $M$ with $K\nleqslant L$. Let $t\in M$ be a projection with kernel $K$.

Define
\[
W=\sum\lset{\ker(v)}{v\in M,\ \im(v)=L}.
\]
Then we claim that $W=V$. Obviously $W\neq0$, so to show that $W=V$ it suffices to show that $W$ is invariant for $M$. We do this using \cref{invlem}: we show that if $u$ is a projection in $M$ and that $W\nleqslant\im(u)$, then $\ker(u)\ls W$. Since $W\nleqslant\im(u)$ there is an element $v\in M$ such that $\im(v)=L$ and $\ker(v)\not\leqslant\im(u)$. This last condition means that $vu$ has nullity $1$, so that its image is $\im(v)=L$ and its kernel is $\ker(u)$. But then $\ker(u)\ls W$, as required.

So $W=V$, which means in particular that $W\nleqslant\im(t)$. So there is $v\in M$ such that $\im(v)=L$ and $\ker(v)\nls\im(t)$. This then means that $vt$ has nullity $1$, and therefore has kernel $\ker(t)$ and image $\im(v)=L$.

So $M$ contains an element $x=vt$ with kernel $K$ and image $L$. This element acts as an invertible linear map on $L$, and this linear map has finite order because $M$ is finite. So there is $a>0$ such that $x^a$ acts as the identity on $L$. But then $x^a$ is the projection $\pwi KL$.
\end{pf}

This leads to the following statement for \cor \prms.

\begin{cory}\label{credconds}
Suppose $M$ is a finite \cor \prm on $V$, and let $S$ be the sum of the kernels of $M$ and $T$ the intersection of the images of $M$. Then $V=S\oplus T$, and $M$ is \cmp.
\end{cory}

\begin{pf}
If $M$ is irreducible and non-trivial, the result follows from \cref{irrconds}. If $M$ is trivial, then the result is immediate. So suppose $M$ is reducible. Then by complete reducibility we can write $V=W_1\oplus W_2$, where $W_1$ and $W_2$ are invariant for $M$. As explained above, we can partition the set of projections in $M$ as $P_1\sqcup P_2$, where $\ker(p)\ls W_i$ and $W_{3-i}\ls\im(p)$ for all $p\in P_i$. Let $M_i$ be the submonoid generated by $P_i$.

Suppose first that one of $P_1,P_2$ is empty; say $P_1=P$ and $P_2=\emptyset$, so that $M=M_1$. Then $M$ acts as a \prm on $W_1$, and acts trivially on $W_2$. In particular, $W_2\ls T$. By induction on dimension we can assume that the \lcnamecref{credconds} holds for $M$ acting on $W_1$, so that $W_1=S\oplus(T\cap W_1)$, giving $V=S\oplus T$. Clearly also $M$ is \cmp in this case.

Now assume instead that both $M_1$ and $M_2$ are non-trivial. For $i=1,2$ let
\[
S_i=\sum_{p\in P_i}\ker(p),
\qquad
T_i=\bigcap_{p\in P_i}\im(p)
\]
Then $S=S_1+S_2$ and $T=T_1\cap T_2$. By induction on $\card M$ we can assume that the \lcnamecref{credconds} holds for $M_1$ and $M_2$, so that $V=S_1\oplus T_1=S_2\oplus T_2$. Clearly $S_2\ls W_2\ls T_1$, so that
\begin{align*}
V&=S_1+T_1
\\
&=S_1+(T_1\cap(S_2+T_2))
\\
&=S_1+S_2+(T_1\cap T_2)\qquad\text{since $S_2\ls T_1$}
\\
&=S+T,
\\
\intertext{while}
S\cap T&=(S_1+S_2)\cap T_1\cap T_2
\\
&=((S_1\cap T_1)+S_2)\cap T_2\qquad\text{since $S_2\ls T_1$}
\\
&=S_2\cap T_2
\\
&=0
\end{align*}
giving $V=S\oplus T$.

To show that $M$ is \cmp, suppose $K\in\kerm M$ and $L\in\imm M$ with $K\nls L$. Take projections $p,q\in M$ such that $K=\ker(p)$ and $L=\im(q)$. Then $p\in P_i$ and $q\in Q_j$ for some $i,j\in\{1,2\}$. Now $i$ must equal $j$, since if $i\neq j$ then $K=\ker(p)\ls W_i\ls\im(q)=L$, a contradiction. So (from the assumption that the \lcnamecref{credconds} holds for $M_i$) the projection $\pwi KL$ lies in $M_i$, and hence in $M$.
\end{pf}

Here is a converse statement.

\begin{thm}\label{crconv}
Suppose $M$ is a finite \prm on $V$, and let $S$ be the sum of the kernels of $M$ and $T$ the intersection of the images of $M$. If $V=S\oplus T$ and $M$ is \cmp, then $M$ is \cor.
\end{thm}

\begin{pf}
Take a subspace $W$ which is invariant for $M$. By \cref{invlem} we can partition the set of projections in $M$ as $P_1\sqcup P_2$, where $\ker(p)\ls W$ for each $p\in P_1$, and $W\ls\im(p)$ for each $p\in P_2$. For $i=1,2$ let $S_i$ be the sum of the kernels of the elements of $P_i$, and $T_i$ the intersection of the images of the elements of $P_i$. Then by definition $S_1\ls W\ls T_2$.

Now we claim that $S_2\ls T_1$. Suppose for a contradiction that $S_2\nls T_1$. Then there are projections $p\in P_1$ and $q\in P_2$ such that $\ker(q)\nls\im(p)$. The assumption that $M$ is \cmp means that $M$ contains the projection $r=\pwi{\ker(q)}{\im(p)}$. So $r$ belongs to $P_1$ or $P_2$. But if $r\in P_1$, then $\ker(q)=\ker(r)\ls W\ls\im(q)$, while if $r\in P_2$, then $\ker(p)\ls W\ls\im(r)=\im(p)$. Either way, we have our contradiction.

Next we claim that $S_2\cap W=0$. To see this, we observe that
\begin{align*}
0&=S\cap T
\\
&=(S_1+S_2)\cap T_1\cap T_2
\\
&=T_1\cap((S_2\cap T_2)+S_1)\qquad\text{because $S_1\ls T_2$}
\\
&=(S_1\cap T_1)+(S_2\cap T_2)\qquad\text{because $S_2\ls T_1$}
\\
&\gs S_2\cap T_2
\\
&\gs S_2\cap W.
\end{align*}

Dually, we can show that $V=T_1+W$. This, together with the previous two claims, shows that there is a subspace $X$ such that $S_2\ls X\ls T_1$, and $V=W\oplus X$. The fact that $S_2\ls X\ls T_1$, together with \cref{invlem}, shows that $X$ is invariant for $M$, and we are done.
\end{pf}

We end this section by considering \srms and complete reducibility. Recall that for a \srm $M$ we define the \emph{projection submonoid} $M_0$ and the \emph{reflection subgroup} $M_1$.

\begin{cory}\label{corsrm}
Suppose $M$ is a finite \cor \srm on $V$. Then $M_0$ is \cor.
\end{cory}

\begin{pf}
It suffices to prove that $V$ is the sum of its minimal $M_0$-invariant subspaces. Since $M$ is \cor, $V$ is the sum of its minimal $M$-invariant subspaces. So we just need to show that any minimal $M$-invariant subspace $W$ is a sum of minimal $M_0$-invariant subspaces. Certainly $W$ is $M_0$-invariant, so it contains a minimal $M_0$-invariant subspace $U$. Because $M_1$ normalises $M_0$, the set of $M_0$-invariant subspaces is preserved by the reflection group $M_1$, so the images of $U$ under $M_1$ are also minimal $M_0$-invariant subspaces. So the sum $X$ of these images is both $M_0$-invariant and $M_1$-invariant, so is $M$-invariant. Clearly $X\ls W$, because $W$ is $M_1$-invariant. So the minimality of $W$ means that $W=X$, and therefore $W$ is a sum of minimal $M_0$-invariant subspaces.
\end{pf}

\begin{rmk}
It is not generally the case that if a \srm $M$ is \cor then the reflection subgroup of $M$ must be. We will see an example of this in the next section.
\end{rmk}

\section{Two examples}\label{exsec}

In this section we construct two general examples of irreducible finite \prms. These are defined over any field $\bbf$, and in any dimension. For these examples we prove irreducibility, find the cardinality and the normalising reflection group, and classify minimal generating sets of projections.

We use $\lan\ \ran$ to denote the $\bbf$-linear span of a vector or set of vectors.

\subsection{Type $A$}\label{typeasec}

\subsubsection{Definition, irreducibility and cardinality}

Take $n\gs1$, and suppose $V$ is a vector space over $\bbf$ of dimension $n$. Take a vector space $V^+$ of dimension $n+1$ containing $V$, and choose a basis $\{e_0,\dots,e_n\}$ of $V^+$ such that $e_i-e_j\in V$ for all $i,j$.

Now given $i,j\in\{0,\dots,n\}$ with $i\neq j$, define $p_{ij}:V\to V$ to be the projection with kernel $\lan e_i-e_j\ran$ and image $\lspan{e_k-e_l}{k,l\neq i}$. Define $\mna n$ to be the monoid generated by the projections $p_{ij}$. Then $\mna n$ is well-defined up to equivalence.

We remark that if $n\gs3$, then $\mna n$ is not self-dual. To see this, observe that the kernels of a \prm naturally correspond to the images of its dual. So a self-dual \prm must have the same number of images as kernels. $\mna n$ has $n+1$ images and $\binom{n+1}2$ kernels when $n\gs2$, so is not self-dual if $n\gs3$.

However, $\mna1$ and $\mna2$ are self-dual. For $\mna1$ this is is easy to see. For $\mna2$, consider the dual space $(V^+)^\ast$ and let $\{\ep_0,\ep_1,\ep_2\}$ be the basis dual to $\{e_0,e_1,e_2\}$. Then the dual space $V^\ast$ can be realised as the quotient of $(V^+)^\ast$ by the subspace spanned by $\ep_0+\ep_1+\ep_2$. Now the map $(V^+)^\ast\to V$ given by linear extension of
\[
\ep_0\longmapsto e_1-e_2,\qquad \ep_1\longmapsto e_2-e_0,\qquad \ep_2\longmapsto e_0-e_1
\]
yields a linear bijection $V^\ast\to V$ under which the projection $p_{ij}$ corresponds to the projection $p_{ij}^\ast$.

\smallskip
Now we examine the basic properties of $\mna n$.

\begin{propn}\label{mnairr}
$\mna n$ is irreducible.
\end{propn}

\begin{pf}
Suppose $W\ls V$ is invariant for $\mna n$; then we must show that $W=0$ or $W=V$. Let $G$ be the graph with vertex set $\{0,\dots,n\}$ and an edge from $i$ to $j$ \iff $e_i-e_j\in W$. Then clearly each connected component of $G$ is a complete graph. If $G$ is the complete graph then $W=V$ and we are done, so assume otherwise. Then we claim that every connected component $C$ of $G$ consists of a single vertex. To see this, take $i\in C$ and $j\notin C$. Then there is no edge from $i$ to $j$ in $G$, so $e_i-e_j\notin W$; that is, $\ker(p_{ij})\nls W$. Now \cref{invlem} gives $W\ls\im(p_{ij})=\lspan{e_k-e_l}{k,l\neq i}$. In other words, $W$ does not contain $e_i-e_k$ for any $k\neq i$, so that $i$ is an isolated vertex of $G$, and therefore $C=\{i\}$.

Applying this argument for every connected component, we find that $G$ is the empty graph. So $W$ does not contain any of the kernels of $\mna n$, and therefore is contained in every image of $\mna n$. But it is easy to see that the intersection of the images of $\mna n$ is $0$.
\end{pf}

\begin{propn}\label{ordermna}
$\mna n$ is finite, with
\[
\card{\mna n}=
\begin{cases}
(n+1)^{n+1}-(n+1)!-n+1&\text{ if }\nchar(\bbf)\neq2,
\\
(n+1)^{n+1}-(n+1)!-2^n\binom{n+1}2+\binom n2+1.&\text{ if }\nchar(\bbf)=2.
\end{cases}.
\]
\end{propn}

\begin{pf}
Recall that we view $V$ as a subspace of a vector space $V^+$ with basis $\{e_0,\dots,e_n\}$. If we define $p_{ij}^+:V^+\to V^+$ for each $i,j$ by mapping $e_i\mapsto e_j$, and $e_k\mapsto e_k$ for all $k\neq i$, then $p_{ij}^+$ is a projection. Now if we define $\mnap n$ to be the \prm on $V^+$ generated by the projections $p_{ij}^+$, then $V$ is invariant for $\mnap n$ and $p_{ij}^+$ acts on $V$ as $p_{ij}$.

The basis $B=\{e_0,\dots,e_n\}$ for $V^+$ is mapped to itself by every element of $\mnap n$, with the non-invertible elements of $\mnap n$ yielding non-invertible functions from $B$ to $B$. Obviously any element of $\mnap n$ is determined by its action on $B$, so we can view $\mna n$ as a submonoid of the \emph{full transformation monoid} $F$ of all functions from $B$ to $B$. In fact, we can be more precise: the projections $p_{ij}^+$ correspond to the idempotents in $F$ of defect $1$ (i.e.\ with exactly one non-fixed point), and Howie shows \cite[Theorem I]{howie} that these idempotents generate the subsemigroup consisting of \emph{all} non-bijective elements of~$F$.

So the number of non-identity elements of $\mnap n$ equals the number of non-bijective elements of $F$, which is $(n+1)^{n+1}-(n+1)!$. But $\mna n$ is isomorphic to a quotient of $\mnap n$. To find the order of this quotient, we have to determine when two different functions $f,g:B\to B$ yield the same linear map on $V$. For this we consider separately the cases $\nchar(\bbf)\neq2$ and $\nchar(\bbf)=2$. For a function $f:B\to B$, let $\frf f:V\to V$ be the map obtained by extending linearly to $V^+$ and then restricting to $V$.

First suppose $\nchar(\bbf)\neq2$. Then we claim that $\frf f=\frf g$ \iff $f$ and $g$ are both constant. The `if' is clear, since if $f$ is constant then $\frf f$ is the zero map. For the `only if', suppose that $f\neq g$ and that $f$ is not constant. Then there are $i,j$ such that $f(e_i)\neq f(e_j)\neq g(e_j)$. This yields $\frf f(e_i-e_j)\neq \frf g(e_i-e_j)$, so that $\frf f\neq \frf g$.

Now suppose $\nchar(\bbf)=2$. Again, if $f$ and $g$ are both constant then $\frf f=\frf g$. If there are $i,j$ such that $\card{\{f(e_i),f(e_j),g(e_i),g(e_j)\}}\gs3$, then $\frf f(e_i-e_j)\neq \frf g(e_i-e_j)$. Likewise, if there are $i,j$ such that $f(e_i)\neq f(e_j)$ while $g(e_i)=g(e_j)$, or $f(e_i)=f(e_j)$ while $g(e_i)\neq g(e_j)$, then $\frf f(e_i-e_j)\neq \frf g(e_i-e_j)$. If there are no such $i,j$ and $f$ and $g$ are not both constant, then $f$ and $g$ between them take exactly two values, say $e_k$ and $e_l$, with $f(e_i)=e_k$ \iff $g(e_i)=e_l$. Now $\frf f(e_i-e_j)$ equals $e_k+e_l$ if $f(i)\neq f(j)$, and $0$ otherwise, and similarly for $g$, so that $\frf f=\frf g$.

Now we can find $\card{\mna n}$: there are $(n+1)^{n+1}$ functions $\{0,\dots,n\}\to\{0,\dots,n\}$, and $(n+1)!$ of these are bijective. The non-bijective functions correspond to non-invertible elements of $\mna n$, except that the $n+1$ different constant functions all give the same element (namely, the zero element) of $\mna n$; furthermore, if $\nchar(\bbf)=2$, the $(2^{n+1}-2)\binom{n+1}2$ functions $f$ taking exactly two values occur in pairs giving the same $\frf f$. So the number of non-invertible elements of $\mna n$ is $(n+1)^{n+1}-(n+1)!-n$ if $\nchar(\bbf)\neq2$, or $(n+1)^{n+1}-(n+1)!-n-(2^n-1)\binom{n+1}2$ if $\nchar(\bbf)=2$.
\end{pf}

\subsubsection{Normalising reflections}

Now (following the discussion in \cref{somemore}) we find the \nrg of $\mna n$. For this, we assume $n\gs2$. Clearly an element $r\in\gll V$ normalises $\mna n$ \iff it preserves the set of lines $\lan e_i-e_j\ran$. 
%
It is easy to show that in this case there is an element $r^+\in\gll{V^+}$ which extends $r$ and preserves the set of lines $\lan e_i\ran$. The assumption that $n\gs2$ means that $r^+$ is uniquely defined. Now the fact that $r^+$ preserves the set of lines $\lan e_i-e_j\ran$ means that there is a permutation $\pi$ of $\{0,\dots,n\}$ and an element $\la\in\bbf^\times$ such that $r^+(e_i)=\la e_{\pi(i)}$ for all $i$.

\begin{propn}\label{mnarefs}
Suppose $n\gs2$, and $r$ is a reflection on $V$ which preserves the set of lines $\lan e_i-e_j\ran$ for $i\neq j$. Let $\pi,\la$ be as defined above. Then one of the following occurs.
\begin{enumerate}
\item\label{transp}
$\la=1$ and $\pi$ is a transposition $(i\ j)$.
\item\label{2spec}
$n=2$, $\la=-1$ and $\pi$ is a transposition $(i\ j)$.
\item\label{3spec}
$n=3$, $\la=-1$ and $\pi$ is a double transposition $(i\ j)(k\ l)$.
\item\label{2spec2}
$n=2$, $\la$ is a primitive cube root of unity in $\bbf$ and $\pi$ is a $3$-cycle $(i\ j\ k)$.
\end{enumerate}
\end{propn}

\begin{pf}
%
%
Recall that for a reflection $r$ on $V$ the space $\lset{(r-1)v}{v\in V}$ is $1$-dimensional.

First we suppose that $\pi$ is the identity permutation. In this case each of the vectors $e_i-e_j$ for $i<j$ is an eigenvector for $r$, with eigenvalue $\la$. So $r$ is a scalar multiple of the identity map on $V$, so is not a reflection, a contradiction.

So $\pi$ is a non-trivial permutation. Suppose that $(i_1\ \dots\ i_a)$ is a cycle of $\pi$ with $2\ls a\ls n$. If $a\gs4$, then the vectors $(r-1)(e_{i_1}-e_{i_2})$ and $(r-1)(e_{i_2}-e_{i_3})$ are linearly independent, so that $r$ is not a reflection. So every cycle of $\pi$ has length at most $3$.

Suppose $\pi$ has a cycle $(i\ j\ k)$ of length $3$. Then $n=2$, since otherwise we can choose a fourth integer $l\neq i,j,k$, and then the vectors $(r-1)(e_i-e_l)$ and $(r-1)(e_j-e_l)$ are linearly independent, so $r$ is not a reflection. Now the eigenvalues of $r$ are $\omega\la$ and $\omega^2\la$, where $\omega$ is a primitive cube root of unity in $\bbf$, forcing $\la=\omega$ or $\omega^2$.

Now we assume that every cycle of $\pi$ has length $1$ or $2$. Suppose there are two cycles of length $2$, say $(i\ j)$ and $(k\ l)$. Then $r$ has linearly independent eigenvectors $e_i-e_j$ and $e_k-e_l$ with eigenvalue $-\la$, forcing $\la=-1$. Now we claim that $n=3$: if not, let $m\in\{0,\dots,n\}$ be distinct from $i,j,k,l$. Then the vectors $(r-1)(e_i-e_m)$ and $(r-1)(e_k-e_m)$ are linearly independent, so that $r$ is not a reflection.

We are left with the case where $\pi$ is a transposition $(i\ j)$. Now the eigenvalues of $r$ are $-\la$ (occurring once) and $\la$ (occurring $n-1$ times). We deduce that either $\la=1$, or $\la=-1$ and $n=2$.
\end{pf}

\begin{cory}\label{norma}
The \nrg of $\mna n$ is:
\begin{enumerate}
\item
the symmetric group $\sss{n+1}$, if $n\gs4$, or if $n=3$ and $\nchar(\bbf)=2$;
\item
the hyperoctahedral group $\sss4\times C_2$, if $n=3$ and $\nchar(\bbf)\neq2$;
\item
the dihedral group $D_{12}$, if $n=2$, $\nchar(\bbf)\neq2$ and $\bbf$ contains no primitive cube roots of~$1$;
\item\label{36}
the direct product $D_{12}\times C_3$, if $n=2$, $\nchar(\bbf)\neq2$ and $\bbf$ contains primitive cube roots of~$1$;
\item
the direct product $\sss3\times C_2$ if $n=2$, $\nchar(\bbf)=2$ and $\bbf$ contains primitive cube roots of~$1$;
\item
the symmetric group $\sss3$, if $n=2$, $\nchar(\bbf)=2$ and $\bbf$ does not contain primitive cube roots of~$1$.
\end{enumerate}
\end{cory}

\begin{rmks}
\begin{enumerate}[beginthm]
\item
Note that when $\bbf=\bbc$, the group $D_{12}\times C_3$ appearing in (\ref{36}) is the complex reflection group of type $G(6,2,2)$.
\item
In the case when $n=3$ and $\nchar(\bbf)=2$, the hyperoctahedral group $\sss4\times C_2$ does not have a natural faithful action. The reflections described in part (\ref{3spec}) of \cref{mnarefs} corresponding to double transpositions still exist, but they now lie inside the symmetric group $\sss4$. Effectively, passing to characteristic $2$ replaces the hyperoctahedral group with the quotient $\sss4$ obtained by quotienting out the longest element.
\end{enumerate}
\end{rmks}


Now we can give the example promised in \cref{irrsec} showing that in a \cor \srm the reflection subgroup need not be \cor. Suppose $\bbf$ has characteristic $3$, and consider the projection monoid $\mna2$. By \cref{norma} this is normalised by the reflection subgroup $\sss3$ (permuting the vectors $e_0,e_1,e_2$). Let $M$ be the \srm $\mna2\sss3$. This is irreducible (because $\mna2$ is); but $\sss3$ does not act completely reducibly on $V$ (the subspace spanned by $e_0+e_1+e_2$ is the unique proper non-zero invariant subspace).

\subsubsection{Minimal generating sets}

In the theory of reflection groups it is an important task to consider minimal generating sets of reflections. Here we describe the minimal sets of projections that generate $\mna n$. We do this using results of Howie on minimal generating sets of idempotents for the full transformation semigroup.

We begin with the extended monoid $\mnap n$ on $V^+$ described above. As we have explained, the semigroup consisting of the non-invertible elements is isomorphic to the semigroup $S$ of non-bijective functions $\{0,\dots,n\}\to\{0,\dots,n\}$, with the projections $p^+_{ij}$ corresponding to the idempotents of defect $1$. Write $f_{ij}$ for the function corresponding to $p^+_{ij}$. Given a set $X$ of ordered pairs $(i,j)$, we can construct a directed graph on $\{0,\dots,n\}$ by drawing an arrow $i\to j$ \iff $(i,j)\in X$. Howie \cite{howie2} shows that if $n\gs2$ then $\lset{f_{ij}}{(i,j)\in X}$ is a minimal generating set for $S$ \iff this graph is a strongly connected tournament (that is, for every pair $(i,j)$ there is either an arrow $i\to j$ or $j\to i$ but not both, and there is a directed path from $i$ to $j$). The classification of minimal generating sets of projections for $\mnap n$ immediately follows: the set $\lset{p_{ij}^+}{(i,j)\in X}$ is a minimal generating set \iff the corresponding graph is a strongly connected tournament.

Now we consider $\mna n$. Recall that as a monoid $\mna n$ is isomorphic to a quotient of $\mnap n$. 
We claim that under this quotient map the minimal generating sets of projections in $\mnap n$ correspond to minimal generating sets of projections in $\mna n$. First we single out the special case where $n=\nchar(\bbf)=2$, which can be checked directly. So assume that either $n\gs3$ or $\nchar(\bbf)\neq2$. Clearly if a set of projections generates $\mnap n$, then corresponding set of projections generates $\mna n$. Conversely, take a generating set $P=\lset{p_{ij}}{(i,j)\in X}$ of projections for $\mna n$, and consider the set $P^+=\lset{p_{ij}^+}{(i,j)\in X}$ of projections in $\mnap n$. Given a projection $p_{ij}^+$, its restriction $p_{ij}$ can be written as a product of elements of $P$. Let $p'$ be the corresponding product of elements of $P^+$. Then $p'$ and $p_{ij}^+$ have the same restriction to $V$. From the proof of \cref{ordermna}, and because we are assuming that either $n\gs3$ or $n=2\neq\nchar(\bbf)$, this means that $p'=p_{ij}^+$. So every projection $p_{ij}^+$ can be written as a product of elements of $P^+$, and therefore $P^+$ generates $\mnap n$.

So generating sets of projections in $\mnap n$ correspond to generating sets of projections in $\mna n$, and therefore the same is true for minimal generating sets. As a consequence of this and Howie's theorem, we deduce the following.

\begin{propn}\label{graphgen}
Suppose $n\gs2$ and $P\subseteq\lset{p_{ij}}{i\neq j}$ is a set of projections in $\mna n$. Then $P$ is a minimal generating set for $\mna n$ \iff the corresponding directed graph on $\{0,\dots,n\}$ is a strongly connected tournament.
\end{propn}

\subsection{Type $B$}\label{typebsec}

\subsubsection{Definition, irreducibility and cardinality}

Now we introduce our next family of examples. Take $n\gs1$, and suppose $V$ is a vector space over $\bbf$ of dimension $n$. Choose a basis $\{e_1,\dots,e_n\}$ for~$V$, and a finite subgroup $Z\ls\bbf^\times$. Then $Z$ is the set of $t$th roots of unity in $\bbf^\times$ for some $t\gs1$.

Given $i\in\{1,\dots,n\}$, let $p_i$ be the projection with kernel $\lan e_i\ran$ and image $\lspan{e_k}{k\neq i}$. Given $i,j\in\{1,\dots,n\}$ with $i\neq j$ and given $z\in Z$, define $p_{ij}^z$ to be the projection with kernel $\lan e_i-z e_j\ran$ and image $\lspan{e_k}{k\neq i}$. Define $\mnb n^t$ to be the monoid generated by the projections $p_i$ and $p_{ij}^z$ for all $i,j,z$. Then $\mnb n^t$ is well-defined up to isomorphism.

\begin{propn}\label{mbnirr}
$\mnb n^t$ is irreducible.
\end{propn}

\begin{pf}
Suppose $W\ls V$ is invariant for $\mnb n^t$; then we must show that $W=0$ or $W=V$.

Suppose first that $W$ contains one of the coordinate axes $\lan e_i\ran$. Then $W$ also contains $p_{ij}^1\lan e_i\ran=\lan e_j\ran$ for every $j$. Hence $W=V$.

Alternatively, suppose $W$ does not contain any of the lines $\lan e_i\ran$. Then \cref{invlem} (applied to the projections $p_1,\dots,p_n$) shows that $W$ is contained in all the hyperplanes $\lspan{e_k}{k\neq i}$, so that $W=0$.
\end{pf}

\begin{propn}\label{ordermnb}
$\mnb n^t$ is finite, with
\[
\card{\mnb n^t}=(nt+1)^n-t^nn!+1.
\]
\end{propn}

\begin{pf}
It is straightforward to write down the action of the generating projections on the vectors $e_1,\dots,e_n$:
\[
p_i(e_k)=
\begin{cases}
0&(k=i)
\\
e_k&(k\neq i),
\end{cases}
\qquad
p_{ij}^z(e_k)=
\begin{cases}
z e_j&(k=i)
\\
e_k&(k\neq i).
\end{cases}
\]
As a consequence, the (spanning) set $S=\lset{z e_k}{z\in Z,\ 1\ls i\ls n}\cup\{0\}$ is mapped to itself by every element of $\mnb n^t$, so $\mnb n^t$ is finite.

To find $\card{\mnb n^t}$, we examine exactly which functions from $S$ to $S$ correspond to elements of $\mnb n^t$. Take a non-identity element $m\in\mnb n^t$. Then $m$ is singular, which means that either $m(e_i)=0$ for some $i$, or $m(e_i)=z m(e_j)$ for some $i\neq j$ and $z\in Z$. Let $f:\{e_1,\dots,e_n\}\to S$ be any function with this property; then we show that there is $m\in\mnb n^t$ extending $f$. Let $r(f)$ be the number of values $k$ such that $z e_k$ lies in the image of $f$ for some $z$; then our assumption on $f$ is the same as saying that $r(f)\ls n-1$. Let $s(f)$ be the number of values $k$ such that $f(e_k)=e_k$. Our proof proceeds by downwards induction on $r(f)$, and for fixed $r(f)$ by downwards induction on $s(f)$.

First we deal with the inductive step where $r(f)<n-1$. Choose $i$ such that no element $z e_i$ is in the image of $f$. Define a new function $g:\{e_1,\dots,e_n\}\to S$ by
\[
g(e_k)=
\begin{cases}
e_i&(k=i)
\\
f(e_k)&(k\neq i).
\end{cases}
\]
Then either $r(g)=r(f)+1$ or $r(g)=r(f)$ and $s(g)=s(f)+1$, so by induction $g$ is the restriction to $\{e_1,\dots,e_n\}$ of an element $m\in\mnb n^t$. Now consider $f(e_i)$: if $f(e_i)=0$, then $f$ is the restriction of $p_im$; alternatively, if $f(e_i)=z e_j$ for $j\neq i$, then $f$ is the restriction of $p_{ij}^z m$.

Now we consider the case where $r(f)=n-1$. In this case there is a unique $i$ such that no $z e_i$ lies in the image of $f$. Now there are two possibilities.

\begin{enumerate}
\item\label{bcase0}
Suppose there is $j$ such that $f(e_j)=0$. Then there is a unique such $j$. We consider two subcases.
\begin{enumerate}[ref=\alph*]
\item\label{bcase0a}
Suppose $i\neq j$. Then $f(e_i)=z e_l$ for some $l\neq i$ and $z\in Z$. We define a new function $g$ by
\[
g(e_k)=
\begin{cases}
e_i&(k=i)
\\
f(e_k)&(k\neq i).
\end{cases}
\]
Then $r(g)=n-1$ and $s(g)=s(f)+1$, so by induction $g$ is the restriction of an element $m\in\mnb n^t$. Now $f$ is the restriction of $p_{il}^z m$.
\item\label{bcase0b}
Suppose $i=j$. Suppose there is $l\neq i$ such that $f(e_l)\neq e_l$. Write $f(e_l)=z e_m$, and define a new function $g$ by
\[
g(e_k)=
\begin{cases}
e_i&(k=l)
\\
f(e_k)&(k\neq l).
\end{cases}
\]
Then $r(g)=n-1$ and $s(g)=s(f)$, and by subcase (\ref{bcase0a}) $g$ is the restriction of some $m\in\mnb n^t$. So $f$ is the restriction of $p_{im}^{-z}m$.

So we can assume that $f(e_l)=e_l$ for all $l\neq i$. Now $f$ is the restriction of $p_i$. 
\end{enumerate}
\item
Suppose instead that there is no $j$ for which $f(e_j)=0$. Then there is a unique pair of integers $j\neq j'$ such that $f(e_j)$ is proportional to $f(e_{j'})$
. The approach in this case is similar to case (\ref{bcase0}). Again, we consider two subcases.
\begin{enumerate}[ref=\alph*]
\item\label{bcase2a}
Suppose $i\neq j,j'$. Then $f(e_i)=z e_l$ for some $l\neq i$ and $z\in Z$, and there is no other $k$ for which $f(e_k)\in\lan e_l\ran$. Now define a new function $g$ by
\[
g(e_k)=
\begin{cases}
e_i&(k=i)
\\
f(e_k)&(k\neq i).
\end{cases}
\]
Then $r(g)=n-1$ and $s(g)=s(f)+1$, so by induction $g$ is the restriction of an element $m\in\mnb n^t$. Now $f$ is the restriction of $p_{il}^z m$.
\item
Now suppose that $i$ equals one of $j$ and $j'$, say $i=j'$. If there is $l\neq i,j$ such that $f(e_l)\neq e_l$ and $f(e_l)\notin\lan e_j\ran$, then we can proceed as in case \ref{bcase0}(\ref{bcase0b}): we write $f(e_l)=z e_m$, and define a new function $g$ by
\[
g(e_k)=
\begin{cases}
e_i&(k=l)
\\
f(e_k)&(k\neq l).
\end{cases}
\]
Then $r(g)=n-1$ and $s(g)=s(f)$, and by subcase (\ref{bcase2a}) $g$ is the restriction of some $m\in\mnb n^t$. So $f$ is the restriction of $p_{im}^z m$.

So we can assume that for every $k\neq i,j$ either $f(e_k)=e_k$ or $f(e_k)\in\lan e_j\ran$. If in fact $f(e_k)=e_k$ for all $k\neq i,j$, then $f(e_i)=z e_j$ and $f(e_j)=y e_j$ for some $y,z\in Z$. Now $f$ is the restriction of $p_{ij}^z p_{ji}^{yz^{-1}}$.

Finally, suppose there is $l\neq i,j$ such that $f(e_l)\in\lan e_j\ran$, while $f(e_k)=e_k$ for $k\neq i,j,l$. Then there are $x,y,z\in Z$ such that $f(e_i)=z e_l$, $f(e_j)=ye_l$, $f(e_l)=xe_j$. Then $f$ is the restriction of $p^z_{il}p^x_{lj}p^{yz^{-1}}_{ji}$.
\end{enumerate}
\end{enumerate}
We have proved our claim, which allows us to complete the calculation of $\card{\mnb n^t}$: there are $(nt+1)^n$ functions from $\{e_1,\dots,e_n\}$ to $S$. The non-invertible elements of $\mnb n^t$ correspond to those for which either $f(e_i)=0$ for some $i$, or $f(e_i)=z e_j$ for some $i\neq j$ and $z\in Z$. This rules out $t^nn!$ functions, so there are $(nt+1)^n-t^nn!$ non-invertible elements of $\mnb n^t$, as required.
\end{pf}

\subsubsection{Normalising reflections}

Now we find the \nrg of $\mnb n^t$. Let $Z^+$ be the set of $(2t)$th roots of unity in~$\bbf$.

\begin{propn}\label{mnbrefs}
Suppose $r$ is a reflection on $V$. Then $r$ normalises $\mnb n^t$ \iff one of the following occurs:
\begin{itemize}
\item
there is $i\in\{1,\dots,n\}$ and $\la\in Z\sm\{1\}$ such that
\[
r(e_i)=\la e_i,\qquad r(e_k)=e_k\text{ for $k\neq i$;}
\]
\item
there are $i\neq j$ and $\la\in Z^+$ such that
\[
r(e_i)=\la e_j,\qquad r(e_j)=\la^{-1} e_i,\qquad r(e_k)=e_k\text{ for $k\neq i,j$.}
\]
\end{itemize}
\end{propn}

\begin{pf}
It is easy to check that the reflections of the given types normalise $\mnb n^t$, so the `if' part is done. So we concentrate on the `only if' part, and assume that $r$ normalises $\mnb n^t$. In particular, $r$ fixes $\kerm{\mnb n^t}$ and $\imm{\mnb n^t}$.

We assume that $n\gs2$. We partition $\kerm{\mnb n^t}$ into the following sets of lines:
\begin{align*}
\call&=\lset{\lan e_i\ran}{1\ls i\ls n},
\\
\calm&=\lset{\lan e_i-z e_j\ran}{1\ls i<j\ls n,\ z\in Z}.
\end{align*}
Then $r$ preserves $\call$, since the lines in $\call$ are precisely the lines which occur as intersections of $n-1$ images of $\mnb n^t$. Let $\pi$ be the permutation of $\{1,\dots,n\}$ corresponding to the action of $r$ on the lines $\lan e_1\ran,\dots,\lan e_n\ran$.

If $\pi$ is the identity permutation, then each $e_i$ is an eigenvector for $r$. Since $r$ is a reflection, this means that there is $i\in\{1,\dots,n\}$ and $\la\neq1$ such that
\[
r(e_i)=\la e_i,\qquad r(e_k)=e_k\text{ for $k\neq i$.}
\]
Now take $j\neq i$. Then $r$ maps the line $\lan e_i+e_j\ran$ to the line $\lan \la e_i+e_j\ran$. So the latter lies in $\calm$, and therefore $\la\in Z$.

Now suppose $\pi$ is not the identity, and let $(i_1\ \dots\ i_a)$ be a cycle of $\pi$, with $a\gs2$. If $a\gs3$, then the vectors $(r-1)e_{i_1}$ and $(r-1)e_{i_2}$ are linearly independent, contradicting the assumption that $r$ is a reflection; so $a=2$. Moreover, $r(e_j)=e_j$ for every $j\neq i_1,i_2$, since otherwise $(r-1)e_{i_1}$ and $(r-1)e_j$ are linearly independent. Now there are $\la_1,\la_2\in\bbf^\times$ such that $r(e_{i_1})=\la_1e_{i_2}$ and $r(e_{i_2})=\la_2e_{i_1}$. The assumption that $r$ is a reflection then forces $\la_2=\la_1^{-1}$. Now $r$ maps the line $\lan e_{i_1}+e_{i_2}\ran$ to $\lan\la_1^{-1}e_{i_1}+\la_1 e_{i_2}\ran$. This line must lie in $\calm$, so that $\la_1^2\in Z$, and therefore $\la_1\in Z^+$.
\end{pf}

This allows us to describe the \nrg of $\mnb n^t$. If we write each element as a matrix with respect to the basis $\{e_1,\dots,e_n\}$, then the matrices obtained are precisely those for which
\begin{itemize}
\item
each row and each column contains exactly one non-zero entry, which is a $2t$th root of unity;
\item
the product of the non-zero entries is a $t$th root of unity.
\end{itemize}

If $\bbf=\bbc$, then this is precisely the complex reflection group of type $G(2t,2,n)$. In general, the \nrg has order $2^{n-1}t^nn!$ if $\bbf$ contains primitive $2t$th roots of unity, or $t^nn!$ otherwise.

\subsubsection{Minimal generating sets}

In this section we classify the minimal generating sets of projections in $\mnb n^t$. The end result is similar to \cref{graphgen}, though here we cannot rely on an external result.

Suppose we have a set $P$ of projections in $\mnb n^t$. We draw a directed graph $G_P$ on $\{1,\dots,n\}$ (where we allow multiple arrows) by drawing an arrow $i\to j$ for each $z\in Z$ such that $p_{ij}^z\in P$. Now we have the following result.

\begin{lemma}\label{gengp}
Suppose $P$ is a set of projections in $\mnb n^t$, and draw the graph $G_P$ as above. Then $P$ generates $\mnb n^t$ \iff the following conditions are all satisfied:
\begin{itemize}
\item
$p_i\in P$ for all $i$;
\item
for each $i\neq j$ and each $z\in Z$ at least one of $p_{ij}^z$ and $p_{ji}^{z^{-1}}$ lies in $P$;
\item
$G_P$ is strongly connected.
\end{itemize}
\end{lemma}

\begin{pf}
Suppose first that $P$ satisfies the three given conditions. Then we must show that the monoid generated by $P$ contains every projection in $\mnb n^t$. For the projections $p_i$ this is given, so take a projection $p_{ij}^z$, and suppose $p_{ij}^{z^{-1}}\in P$. Take a directed path $i_0\to \dots\to i_r$ in $G_P$, with $i_0=i$ and $i_r=j$. Then there are $z_1,\dots,z_r\in Z$ such that the projection $p(a)=p_{i_{a-1}i_a}^{z_a}$ lies in $P$ for $a=1,\dots,r$. Now consider the composition
\[
q=p(1)\dots p(r)p_{ji}^{z^{-1}}.
\]
This has rank $n-1$, since each projection has rank $n-1$, and $\ker(p(a))\nls\im(p(a+1))$ for each $a$, and similarly $\ker(p(r))\nls\im(p_{ji}^{z^{-1}})$. So
\[
\ker(q)=\ker(p_{ji}^{z^{-1}})=\ker(p_{ij}^z),\qquad\im(q)=\im(p(1))=\im(\ker(p_{ij}^z).
\]
Since $\mnb n^t$ is finite, some power of $q$ equals $p_{ij}^z$, so $p_{ij}^z$ lies in the monoid generated by $P$.

Conversely, suppose $P$ generates $\mnb n^t$. Then every kernel of $\mnb n^t$ must occur as the kernel of some element of $P$. To see this, take a projection $p\in\mnb n^t$, and write $p=p(1)\dots p(r)$, where $p(1),\dots,p(r)\in P$. Then $\ker(p)\gs\ker(p(r))$; but both kernels have dimension $1$, giving $\ker(p)=\ker(p(r))$, so $\ker(p)$ is the kernel of an element of $P$.

In particular this means that the projections $p_1,\dots,p_n$ all lie in $P$, since $p_i$ is the only projection in $\mnb n^t$ with kernel $\lan e_i\ran$. Furthermore, given $i\neq j$ and $z\in Z$, the only projections in $\mnb n^t$ with kernel $\lan e_i-z e_j\ran$ are $p_{ij}^z$ and $p_{ji}^{z^{-1}}$, so at least one of these lies in $P$.

It remains to show that $G_P$ is strongly connected. So take $i\neq j$; then we will show that there is a directed path from $i$ to $j$ in $G_P$. Because $P$ generates $\mnb n^t$, the projection $p_{ij}^1$ can be written as a product $p(1)\dots p(r)$, where $p(1),\dots,p(r)\in P$. Now define an \emph{$ij$-tuple} to be a tuple $(p(1),\dots,p(r))$ of elements of $P$ such that $\ker(p(1)\dots p(r))=\ker(p_{ij}^1)$ and $\im(p(1)\dots p(r))=\im(p_{ij}^1)$. (This is equivalent to saying that $\im(p(1))=\im(p_{ij}^1)$, $\ker(p(r))=\ker(p_{ij}^1)$, and $\ker(p(a))\nls\im(p(a+1))$ for $1\ls a<r$.)

Choose an $ij$-tuple $(p(1),\dots,p(r))$ with $r$ as small as possible. Then we claim that none of the $p(a)$ can be of the form $p_k$. To see this, first note that $p(r)\neq p_k$ since $\ker(p(r))=\lan e_i-e_j\ran\neq\ker(p_k)$. Next observe that if $a<r$ and $p(a)=p_k$, then $\ker(p(a))=\lan e_k\ran$, so that $\im(p(a+1))$ must equal $\lspan{e_l}{l\neq k}=\im(p(a))$; but then $(p(1),\dots,p(a-1),p(a+1),\dots,p(r))$ is also an $ij$-tuple, contradicting minimality.

So each of $p(1),\dots,p(r)$ is one of the projections $p_{st}^z$. Taking $1\ls a<r$ and writing $p(a)=p_{st}^z$ and $p(a+1)=p_{uv}^y$, the fact that $\ker(p(a))\nls\im(p(a+1))$ gives $u=s$ or $u=t$. But if $u=s$ then $\im(p(a))=\im(p(a+1))$ and again minimality is contradicted.

So we can conclude that there are $i_0,\dots,i_r$ with $i_{a-1}\neq i_a$ for each $a$, and $z_1,\dots,z_r\in Z$ such that $p(a)=p_{i_{a-1}i_a}^{z_a}$ for each $a$. The fact that $\im(p(1))=\im(p_{ij}^1)$ is the same as saying $i_0=i$, while the fact that $\ker(p(r))=\ker(p_{ij}^1)$ gives $\{i_{r-1},i_r\}=\{i,j\}$. So we have a directed path $i\to i_1\to\dots\to i_r$ in $G_P$, with either $i_{r-1}$ or $i_r$ equal to $j$, and we are done.
\end{pf}

Having classified generating sets of projections in $\mnb n^t$, we need to find which of these are minimal. For this, we need a little more graph theory. 
We do not claim originality for the following \lcnamecref{tournlem}, but it is easier to find a proof than to find a reference.

\begin{lemma}\label{tournlem}
Suppose $n\gs3$ and $G$ is a strongly connected directed graph on $\{1,\dots,n\}$ with at least one arrow $i\to j$ or $j\to i$ for all $i\neq j$. If for given $i,j$ there are arrows $i\to j$ and $j\to i$, then one of these arrows  can be deleted with the resulting graph still being strongly connected.
\end{lemma}

\begin{pf}
Suppose there is a pair $i,j$ for which this is not true; then there must be a unique arrow $i\to j$ and a unique arrow $j\to i$ in $G$, and every directed path from $i$ to $j$ involves the arrow $i\to j$, and every directed path from $j$ to $i$ involves the arrow $j\to i$. This implies in particular that if $k\neq i,j$ then there cannot be arrows $i\to k\to j$, and there cannot be arrows $j\to k\to i$.

So the set $\{1,\dots,n\}\sm\{i,j\}$ can be partitioned as $K\sqcup L$, where for every $k\in K$ there are arrows $i\to k$ and $j\to k$ but no arrows $k\to i$ or $k\to j$, while for every $l\in L$ there are arrows $l\to i$ and $l\to j$, but no arrows $i\to l$ or $j\to l$.

Since $n\gs3$, at least one of $K$ and $L$ is non-empty; \wolog suppose $K$ is non-empty. Since $G$ is strongly connected, there must be an arrow $k\to l$ for some $k\in K$ and $l\notin K$. But then $l\in L$, and there are arrows $i\to k\to l\to j$, contrary to hypothesis.
\end{pf}

Now we can classify minimal generating sets when $n\gs3$.

\begin{propn}\label{mintypeb}
Suppose $n\gs3$ and $P$ is a set of projections in $\mnb n^t$, and draw the graph $G_P$ as above. Then $P$ is a minimal generating set for $\mnb n^t$ \iff:
\begin{itemize}
\item
$p_i\in P$ for all $i$;
\item
for each $i\neq j$ and each $z\in Z$ exactly one of $p_{ij}^z$ and $p_{ji}^{z^{-1}}$ lies in $P$;
\item
$G_P$ is strongly connected.
\end{itemize}
\end{propn}

\begin{pf}
Suppose $P$ satisfies the three given conditions. Then by \cref{gengp} $P$ is a generating set for $\mnb n^t$. Furthermore, $P$ is minimal, since if any element of $P$ is deleted, then (again by \cref{gengp}) the resulting set does not generate $\mnb n^t$.

Conversely, suppose $P$ is a generating set but does not satisfy all three given conditions. Then there must be some $i,j,z$ such that both $p_{ij}^z$ and $p_{ji}^{z^{-1}}$ lie in $P$. Now we claim that we can remove one of $p_{ij}^z$ and $p_{ji}^{z^{-1}}$ from $P$ with the resulting set still being a generating set. The graph $G_P$ satisfies the hypotheses of \cref{tournlem}, and includes arrows $i\to j$ and $j\to i$. So by \cref{tournlem} one of these arrows, say an arrow from $i\to j$, can be deleted from $G_P$ with the resulting graph $G_P^-$ still being strongly connected. Now let $Q=P\sm\{p_{ij}^z\}$. Then $G_Q=G_P^-$ is strongly connected, so that $Q$ is a generating set and $P$ is not minimal.
\end{pf}

\section{Classification of irreducible \prms on $\bbc^2$}\label{classsec}

Our aim in this section is to classify finite irreducible \prms on $\bbc^2$. In fact, our results yield a classification over any field of characteristic zero, but for simplicity we will work entirely with $\bbc$ from \cref{someexsec} onwards.

\subsection{The trace group}

In this subsection we work over an arbitrary field $\bbf$. Suppose $M$ is a \prm on $\bbf^2$, take two non-identity elements $p,q\in M$, and suppose $\im(p)\neq\ker(q)$. Then $pq$ equals $\tr{pq}$ times the projection $\pwi{\ker(q)}{\im(p)}$. So if $a\in\bbn$, then $(pq)^a$ equals $\tr{pq}^a$ times $\pwi{\ker(q)}{\im(p)}$. Since $M$ is finite, this means that $\tr{pq}$ is zero or a root of unity in $\bbf$; we will call this the ``trace condition''. If we define the \emph{trace group} of $M$ to be the group $G$ generated by all the non-zero traces of the non-invertible elements of $M$, then $G$ is a finite subgroup of $\bbf^\times$.

The next result shows how the trace group enables us to find the cardinality of an irreducible \prm on $\bbf^2$.

\begin{propn}\label{cardprm}
Suppose $M$ is an irreducible \prm on $\bbf^2$. Then the number of non-zero elements in $M$ is $kl\card G+1$, where $k=\card{\kerm M}$, $l=\card{\imm M}$ and $G$ is the trace group of $M$. Furthermore, $0$ is an element of $M$ \iff $\kerm M\cap\imm M\neq\emptyset$.
\end{propn}

\begin{pf}
We prove two claims.

\clamno
If $K\in\kerm M$ and $L\in\imm M$, then there is an element of $M$ with kernel $K$ and image~$L$.
\prof
In the case where $K\neq L$ this follows from \cref{irrconds}: $M$ must contain $\pwi KL$. Even in the case where $K=L$, the argument from the proof of \cref{irrconds} still works to show the existence of an element with kernel and image both equal to $K$.
\malc

\clamno
If $m$ is a non-zero non-invertible element of $M$ and $g\in\bbf^\times$, then $gm\in M$ \iff $g\in G$.
\prof
For each projection $p\in M$, let $G_p=\lset{g\in\bbf^\times}{gp\in M}$. Then $G_p\ls G$. Since any non-invertible element of $M$ with non-zero trace $t$ equals $t$ times a projection, the trace group $G$ is the union of the groups $G_p$ for all projections $p\in M$.

To prove the claim, first consider the case where $m$ is a projection; then we want to show that $G_m=G$. If $G_m<G$, then there is an element $g\in G\sm G_m$ and another projection $p$ such that $g\in G_p$. Because $M$ is irreducible, it has at least two images and at least two kernels, so we can choose a kernel $K\neq\im(p)$ and an image $L\neq\ker(m)$. By Claim 1 we can find an element $r\in M$ with kernel $K$ and image $L$. Similarly, we can find a non-zero element $s$ with $\ker(s)\neq\im(m)$ and $\im(s)\neq\ker(p)$. This means that the element $mrpsm$ is non-zero, and hence has image $\im(m)$ and kernel $\ker(m)$, so equals $hm$ for some $h\in G_m$. Additionally, $M$ contains the element $mr(gp)sm$, which equals $ghm$. But then $gh\in G_m$, so that $g\in G_m$, a contradiction. So $G_m=G$.

Next consider the case where $m$ is a scalar multiple of a projection; say $m=gm'$ for some projection $m'\in M$. The argument in the paragraph shows that the claim holds for $m'$, which immediately shows that it holds for $m$.

Finally consider the case where $m$ is not a scalar multiple of a projection; this means that $\im(m)=\ker(m)$. Let $u\in M$ be any projection with $\im(u)=\im(m)$, and let $v\in M$ be a projection with $\ker(v)\neq\im(m)\neq\im(v)$. Then $um$ is a non-zero scalar multiple of $m$. More generally, $gum$ is a non-zero scalar multiple of $m$ for any $g\in G$, so there are at least $\card G$ different multiples of $m$ in $M$. On the other hand, if $g\in\bbf^\times$ with $gm\in M$, then $gmv$ is a scalar multiple of the projection $r=\pwi{\ker(v)}{\im(m)}$. So (because the claim holds with $r$ in place of $m$) there are most $\card G$ scalar multiples of $m$ in $M$. So equality holds, and the scalar multiples of $m$ in $M$ are precisely the elements $gum$ for $g\in G$. One of these is $m$ itself, and therefore these scalar multiples are precisely the elements $gm$ for $g\in G$.
\malc
Now we use these two claims to prove the first statement in the \lcnamecref{cardprm}. Suppose $m$ is a non-zero non-invertible element of $M$. Then $M$ has rank $1$, so $\ker(m)$ and $\im(m)$ are both $1$-dimensional, and $\ker(m)\in\kerm M$ and $\im(m)\in\imm M$. On the other hand, given $K\in\kerm M$ and $L\in\imm M$, Claims 1 and 2 show that there are exactly $\card G$ elements of $M$ with kernel $K$ and image $L$. So there are $kl\card G$ non-zero non-invertible elements in $M$, which gives the result.

The second statement in the \lcnamecref{cardprm} is straightforward to prove: suppose $p_1,\dots,p_r$ are projections in $M$. Then $p_1\dots p_r=0$ \iff $\ker(p_i)=\im(p_{i+1})$ for some $i$. So it possible to write $0$ as a product of projections in $M$ \iff $\kerm M\cap\imm M\neq\emptyset$.
\end{pf}

The next result makes it much easier to determine the trace group. Recall the \prm $\mna2$ from \cref{typeasec}; this can be characterised as the unique (up to equivalence) \cmp \prm $M$ with $\kerm M=\imm M$ and $\card{\kerm M}=3$.

\begin{propn}\label{pwisetrace}
Suppose $M$ is a finite irreducible \prm on $\bbf^2$, and let $P$ be the set of projections in $M$. If $M$ is not equivalent to $\mna2$, then the trace group of $M$ is generated by the non-zero elements of the set
\[
\lset{\tr{pq}}{p,q\in P}.
\]
\end{propn}

\begin{pf}
In this proof we use the following notation. Given $p,q\in P$ with $\im(p)\neq\ker(q)$, we write $\sx pq$ for the projection $\pwi{\ker(q)}{\im(p)}$. By \cref{irrconds} $M$ is \cmp, and therefore $\sx pq\in P$. Moreover, $pq=\tr{pq}\sx pq$.

Let $S$ be the set given in the \lcnamecref{pwisetrace}. Since any element of $M$ can be written as a product of projections, it suffices to show that any non-zero element $tuv$ with $t,u,v\in P$ can be written in the form $xrs$, where $r,s\in P$ and $x$ is a product of elements of $S$.

The assumption that $t,u,v$ are projections and that $tuv\neq0$ immediately gives
\[
\im(t)\neq\ker(t)\neq\im(u)\neq\ker(u)\neq\im(v)\neq\ker(v).
\]

First suppose $\im(t)\neq\ker(u)$. Then $tuv$ equals $\tr{tu}\sx tuv$, which is what we want. Similarly, if $\im(u)\neq\ker(v)$, then $tuv$ equals $\tr{uv}t\sx uv$, which is what we want.


We are left with the case where $\im(t)=\ker(u)$ and $\im(u)=\ker(v)$. Now we claim that either: 
\begin{enumerate}[label=(\alph*)]
\item
there is an image $L$ of $M$ which is different from $\im(t),\ker(t),\im(u)$, or
\item
there is a kernel $K$ of $M$ which is different from $\ker(u),\im(v),\ker(v)$.
\end{enumerate}
To see this, first suppose that $\ker(t)\neq\im(v)$. Then (a) is satisfied with $L=\im(v)$. Alternatively, if $\ker(t)=\im(v)$, then we have three different lines $\im(t),\ker(t),\im(u)$ each of which is both a kernel and an image of $M$. The assumption that $M$ is not equivalent to $\mna2$ means that there is a fourth line $L$ which is either a kernel or an image of $M$, and so either (a) or (b) is true.

We assume \wolog that (a) is true, and let $r$ be the projection with image $L$ and kernel $\ker(t)$, and $s$ the projection with image $L$ and kernel $\ker(v)$. Then $tr=t$ because $t$ and $r$ have the same kernel, so that
\[
tuv=truv=\tr{ru}t\sx ruv=\tr{ru}\tr{\sx ruv}ts,
\]
which is what we want.
\end{pf}


We remark that the exception in \cref{pwisetrace} for monoids equivalent to $\mna2$ is necessary if $\nchar(\bbf)\neq2$: if $p,q\in\mna2$ are projections, then $\tr{pq}\in\{0,1\}$, but in fact $\mna2$ has trace group $\{\pm1\}$.

\subsection{Some examples in $\bbc^2$}\label{someexsec}

For the rest of \cref{classsec} we concentrate on the case $\bbf=\bbc$. In this subsection  we describe the \prms which will appear in the classification. Each \prm $M$ described will be \cmp, which means that we can specify it by giving $\kerm M$ and $\imm M$. In each case we also give the trace group (which will yield $\card M$ via \cref{cardprm}) and the \nrg of $M$. We also make comments about duality.

We will identify linear endomorphisms of $\bbc^2$ with $2\times2$ matrices (acting on $\bbc^2$ on the left). Given a non-zero vector $\binom xy\in\bbc^2$, we'll write $\spn xy$ for the line it spans.

\subsubsection*{\Prms with only two images}

First, let $S$ be a finite set of roots of unity in $\bbc$ such that $1\in S$, and choose $i\in\{0,1,2\}$ such that $i+\card S\gs2$. Define the \cmp \prms $\as Si$ as follows.
\[
\begin{array}{ccc}\hline
i&\kerm{\as Si}&\imm{\as Si}\\\hline
\\[-10pt]
0&\lset{\spn1s}{s\in S}&\sett{\spn10,\spn01}\\[8pt]
1&\lset{\spn1s}{s\in S}\cup\sett{\spn10}&\sett{\spn10,\spn01}\\[8pt]
2&\lset{\spn1s}{s\in S}\cup\sett{\spn10,\spn01}&\sett{\spn10,\spn01}\\[6pt]\hline
\end{array}
\]
Note that in the case where $S$ is the group of all $2m$th roots of unity, $\as S2$ is just the monoid $\mnb2^m$ from \cref{typebsec}.

It is easy to see from \cref{invlem} (with the assumption that $i+\card S\gs2$) that $\as Si$ is irreducible.

\begin{propn}\label{asi}
The trace group of $\as Si$ is the subgroup of $\bbc^\times$ generated by $S$, and
\[
\card{\as Si}=
\begin{cases}
2\card S\card G+1&(i=0)\\
2(\card S+1)\card G+2&(i=1)\\
2(\card S+2)\card G+2&(i=2).
\end{cases}
\]
\end{propn}

\begin{pf}
To find the trace group, we exploit \cref{pwisetrace}. The projections in $\as Si$ are of the form
\[
\fmx1{-s^{-1}}00,\quad\fmx1000,\quad\fmx00{-s}1,\quad\fmx0001
\]
for $s\in S$, and we can check directly that the trace of any product of two matrices of this form is $0$ or $st^{-1}$ for some $s,t\in S$. Since $1\in S$ by assumption, the non-zero traces generate the same subgroup of $\bbc$ as $S$ does.

Now the second part of the \lcnamecref{asi} follows from \cref{cardprm}.
%
\end{pf}

To find the \nrg of $\as Si$, we introduce the group
\[
G_S=\lset{g\in\bbc^\times}{gs\in S\text{ for all }s\in S}.
\]
$G_S$ is then the group of $m$th roots of unity in $\bbc$, for some $m\in\bbn$.

\begin{propn}\label{asinorm}
\begin{enumerate}[beginthm]
\item\label{hardasinorm}
If $i=0$ or $2$ and there is $t\in\bbc$ such that $t^2s^{-1}\in S$ for all $s\in S$, then the \nrg of $\as Si$ is
\[
\lset{\fmx g00h,\fmx {tg}00{t^{-1}h},\fmx0gh0,\fmx0{tg}{t^{-1}h}0}{g,h\in G_S}.
\]
This is conjugate to the complex reflection group of type $G(2m,2,2)$.
\item\label{basicasinorm}
If $i=1$ or there is no such $t$, then the \nrg of $\as Si$ is
\[
\lset{\fmx g00h}{g,h\in G_S},
\]
which is isomorphic to $C_m\times C_m$.
\end{enumerate}
\end{propn}

\begin{pf}
Any normalising reflection $r$ of $\as Si$ must permute the images and permute the kernels of $\as Si$. For $i=1$ this means that $r$ fixes each of the lines $\spn10$ and $\spn01$, so acts via a diagonal matrix. The fact that $r$ is a reflection that permutes the lines $\spn1s$ then means that the diagonal entries must be $1$ and $g$, for some $g\in G_S$. These reflections generate the group in (\ref{basicasinorm}).

For $i=0,2$ the \nrg may additionally include reflections that interchange the two lines $\spn10$ and $\spn01$. These act via matrices of the form
\[
\fmx0x{x^{-1}}0,
\]
and the condition that such a reflection permutes the lines $\spn1s$ gives $x^2s^{-1}\in S$ for all $s\in S$. This can happen only if $t$ exists as in the statement of the \lcnamecref{asinorm}, and $x\in tG_S$. In this case, these reflections, together with the reflections given by diagonal matrices, give the group in (\ref{hardasinorm}).
\end{pf}

We can ask when two of the monoids $\as Si$ are equivalent: clearly if $\as Si$ and $\as Tj$ are equivalent, then $i=j$ (because $i$ is the number of images of $\as Si$ which are also kernels, and this number is invariant under equivalence). Now equivalences come in two different types.
\begin{itemize}
\item
Take $u\in S$ and let $T=\lset{su^{-1}}{s\in S}$. Then $\as Si$ is equivalent to $\as Ti$ (by rescaling the first coordinate by $u$).
\item
Take $u\in S$ and let $T=\lset{s^{-1}u}{s\in S}$. If $i=0$ or $2$, then $\as Si$ is equivalent to $\as Ti$ (by interchanging the two coordinates and rescaling one coordinate by $u$).
\end{itemize}

We can also consider duality. We will see in \cref{c2classn} that the monoids $\as Si$ are the only irreducible finite \prms with exactly two images, up to equivalence. As a consequence, their duals are the only irreducible finite \prms with exactly two kernels. In the cases where $\as Si$ has exactly two kernels, $\as Si$ is in fact self-dual. There are two cases to consider.
\begin{itemize}
\item
Suppose $i=0$ and $S=\{1,s\}$. Then the linear map given by $\fmx 1{-1}s{-1}$ interchanges the two images $\spn10,\spn01$ with the two kernels $\spn11,\spn1s$, giving an equivalence between $\as S0$ and its dual;
\item
Suppose $i=1$ and $S=\{1\}$. Then the linear map given by $\fmx 1{-1}0{-1}$ interchanges the two images $\spn10,\spn01$ and with the two kernels $\spn11,\spn10$, giving an equivalence between $\as S1$ and its dual.
\end{itemize}

\subsubsection*{A one-parameter family}

Now we describe a one-parameter family of \prms with three images and three kernels. Suppose $w\neq\pm1$ is a root of unity in $\bbc$, and let $\bw w$ be the \prm with
\[
\imm{\bw w}=\sett{\spn11,\spn10,\spn1{1+w}},\qquad\kerm{\bw w}=\sett{\spn11,\spn01,\spn{1-w}1}.
\]
It is immediate from \cref{invlem} that $\bw w$ is irreducible.

First we find the trace group and the \nrg of $\bw w$.

\begin{propn}\label{bw}
Suppose $w\neq\pm1$ is a root of unity in $\bbc$, and let $G$ be the subgroup of $\bbc^\times$ generated by $w$ and $-w$. Then $\bw w$ has trace group $G$, and order $9\card G+2$.
\end{propn}

\begin{pf}
We use \cref{pwisetrace}. $\bw w$ contains eight projections, and it is routine to check that the traces of products of pairs of these projections are $0,1,\pm w^{\pm1},w^{\pm2}$. So the non-zero traces generate $G$. Now the order of $\bw w$ follows from \cref{cardprm}.
\end{pf}

\begin{propn}\label{bwnorm}
The \nrg of $\bw w$ has order $2$, and is generated by the reflection
\[
\mfrac1w\fmx1{w-1}{w+1}{-1}.
\]
\end{propn}

\begin{pf}
Any normalising reflection permutes the kernels of $\bw w$ and the images of $\bw w$, so in particular fixes the line $\spn11$. It is a routine exercise to check that the only reflection that accomplishes this is the one given.
\end{pf}

Now we consider equivalences between the monoids $\bw w$ and their duals.

\begin{lemma}\label{bweqd}
Suppose $w,x\neq\pm1$ are roots of unity in $\bbc$.
\begin{enumerate}
\item
$\bw w$ is equivalent to $\bw x$ \iff $w=x^{\pm1}$.
\item
$\bw w$ is equivalent to the dual of $\bw x$ \iff $w=-x^{\pm1}$.
\end{enumerate}
\end{lemma}

\begin{pf}
\begin{enumerate}[beginthm]
\item
Suppose we have a linear map $a:\bbc^2\to\bbc^2$ which yields an equivalence between $\bw w$ and $\bw x$. Then $a$ fixes the line $\spn11$, sends the two lines $\spn10,\spn1{1+w}$ to $\spn10,\spn1{1+x}$ in some order, and sends the lines $\spn01,\spn{1-w}1$ to $\spn01,\spn{1-x}1$ in some order. This gives four cases:
\begin{enumerate}
\item
$a$ fixes the two lines $\spn10$ and $\spn01$. Then $a$ is just a scalar multiple of the identity, so fixes the two lines $\spn1{1+w}$ and $\spn{1-w}1$, so that $x=w$.
\item
$a$ fixes the line $\spn10$ and sends $\spn{1-w}1$ to $\spn01$. Then (up to scaling) $a$ is given by $\fmx1{w-1}0w$, so sends $\spn01$ to $\spn{1-w^{-1}}1$ and $\spn1{1+w}$ to $\spn1{1+w^{-1}}$. So $x=w^{-1}$.
\item
$a$ fixes the line $\spn01$ and sends $\spn1{1+w}$ to $\spn10$. Then (up to scaling) $a$ is given by $\fmx w0{w+1}{-1}$, so sends $\spn10$ to $\spn1{1+w^{-1}}$ and $\spn{1-w}1$ to $\spn{1-w^{-1}}1$. So again $x=w^{-1}$.
\item
$a$ sends $\spn1{1+w}$ to $\spn10$ and $\spn{1-w}1$ to $\spn01$. Then (up to scaling) $a$ is given by $\fmx1{w-1}{w+1}{-1}$, so sends $\spn10$ to $\spn1{1+w}$ and $\spn01$ to $\spn{1-w}1$, so that $x=w$.
\end{enumerate}
\item
First we note that the map $a=\fmx0110$ yields an equivalence from $\bw w$ to the dual of $\bw{-w}$. Combining this with part (1) yields the result.
\qedhere
\end{enumerate}
\end{pf}

\subsubsection*{Some sporadic examples}

Finally, we describe some sporadic examples $\ci i$ for $i=0,1,2,3,4,5$. We fix a primitive cube root of unity $\om\in\bbc$, and we define \prms as follows.
\[
\begin{array}{cccc}\hline
i&\imm{\ci i}&\kerm{\ci i}&\card{\ci i}\\\hline
0&\sett{\spn10,\spn01,\spn{-1}1}&\sett{\spn10,\spn01,\spn{-1}1}&20\\
1&\sett{\spn10,\spn01,\spn{-1}1}&\sett{\spn10,\spn01,\spn\omega1}&56\\
2&\sett{\spn10,\spn01,\spn{-1}1}&\sett{\spn10,\spn01,\spn\omega1,\spn{\omega^2}1}&74\\
3&\sett{\spn10,\spn01,\spn{-1}1}&\sett{\spn10,\spn01,\spn{-1}1,\spn\omega1}&74\\
4&\sett{\spn10,\spn01,\spn{-1}1}&\sett{\spn10,\spn01,\spn{-1}1,\spn\omega1,\spn{\omega^2}1}&92\\
5&\sett{\spn10,\spn01,\spn{-1}1,\spn\omega1}&\sett{\spn10,\spn01,\spn{-1}1,\spn\omega1}&98\\\hline
\end{array}
\]
Note that $\ci0$ is just the \prm $\mna2$ from \cref{typeasec}. We can easily check that the trace group of $\ci0$ is $\{\pm1\}$, while for $1\ls i\ls 5$ the trace group of $\ci i$ is the group of $6$th roots of unity in $\bbc$. The order of each $\ci i$ now follows from \cref{cardprm}, and is given in the table.

By considering the cardinalities of $\kerm M $, $\imm M$ and $\kerm M\cap\imm M$ we can see that these six \prms are pairwise inequivalent. $\ci0$ and $\ci5$ are self-dual, and $\ci1$ is equivalent to its dual via the linear map $\fmx0\om{-1}0$.

Now we find the \nrg of $\ci i$.

\begin{propn}\label{cinorm}
\begin{itemize}[beginthm]
\item
The \nrg of $\ci0$ and of $\ci4$ is conjugate to the complex reflection group of type $G(6,2,2)$, generated by the reflections
\[
\fmx0110,\quad\fmx110{-1},\quad\fmx0\om{-\om}{-\om}
\]
\item
The \nrg of $\ci1$ is trivial.
\item
The \nrg of $\ci2$ is $C_2\times C_2$, generated by the reflections
\[
\fmx0110,\quad\fmx0{-1}{-1}0.
\]
\item
The \nrg of $\ci3$ is $C_3$, generated by the reflection
\[
\fmx0\om{-\om}{-\om}.
\]
\item
The \nrg of $\ci5$ is conjugate to the exceptional reflection group $G_5$, generated by the reflections
\[
\fmx0\om{-\om}{-\om},\quad\fmx0{-\om^2}1{-\om}.
\]
\end{itemize}
\end{propn}

\begin{pf}
In each case we find the group generated by the reflections that fix the set of kernels and the set of images. In all cases except $\ci5$, this means that our reflections must fix the set of lines $\spn10,\spn01,\spn{-1}1$. It is easy to classify these reflections, and see that they generate a group conjugate to $G(6,2,2)$ (this is really a special case of \cref{mnarefs,norma}). For $i=1,2,3,4$ we then impose further restrictions to yield a subgroup. In the case of $\ci5$ we need to find the reflections that fix the set of lines $\spn10,\spn01,\spn{-1}1,\spn\om1$, which is a routine exercise.
\end{pf}


%
%
%

\subsection{Classification}\label{c2classn}

In this section we show that the \prms in \cref{someexsec} give all the finite irreducible \prms on $\bbc^2$ up to equivalence and duality.

\begin{thm}\label{classify}
Suppose $M$ is an irreducible finite \prm on $\bbc^2$. Then $M$ or its dual is equivalent to one of the \prms $\as Si$, $\bw w$ or $\ci i$.
\end{thm}

The proof splits up into several propositions.

\begin{propn}\label{onlytwo}
Suppose $M$ is a finite irreducible \prm on $\bbc^2$ with $\card{\imm M}=2$. Then $M$ is equivalent to $\as Si$ for some $S,i$.
\end{propn}

\begin{pf}
By replacing $M$ with an equivalent \prm, we can assume that the two images of $M$ are $\spn10$ and $\spn01$. Then $M$ must have a kernel different from $\spn10$ and $\spn01$, since otherwise $M$ is reducible. By rescaling one coordinate we can assume that one of the kernels is $\spn11$. Now $M$ contains the projection
\[
\fmx1{-1}00.
\]
If $\spn1w$ is also a kernel with $w\neq0$, then $M$ also contains the projection
\[
\fmx00{-w}1.
\]
So $M$ also contains
\[
\fmx1{-1}00\fmx00{-w}1=\fmx w{-1}00,
\]
and the trace condition shows that $w$ is a root of unity.

So
\[
\kerm M=\lset{\spn1s}{s\in S}\cup T
\]
for some finite set $S$ of roots of unity with $1\in S$, and for some $T\subseteq\sett{\spn10,\spn01}$. So $M=\as Si$ for some $i$, except in the case where $T=\sett{\spn01}$, in which case $M$ is equivalent to $\as{\bar S}1$ (where $\bar S=\lset{w^{-1}}{w\in S}$).
\end{pf}

Now we can restriction attention to \prms with at least three images, and (by duality) at least three kernels.

\begin{propn}\label{int0}
Suppose $M$ is a finite \prm on $\bbc^2$, and that $a,b,c$ are distinct elements of $\kerm M$, and $d,e,f$ are distinct elements of $\imm M$. Then $\{a,b,c\}\cap\{d,e,f\}\neq\emptyset$.
\end{propn}

\begin{pf}
Suppose for a contradiction that $a,b,c,d,e,f$ are distinct. Without loss of generality we can assume that $a,b,c$ are the lines $\spn10$, $\spn01$ and $\spn1{-1}$, and we write $d=\spn x1$, $e=\spn y1$, $f=\spn z1$ for distinct $x,y,z\neq0,-1$. Now $M$ contains the projections
\[
\fmx10{x^{-1}}0,\ \fmx0y01,\ \fmx{\frac x{x+1}}{\frac x{x+1}}{\frac1{x+1}}{\frac1{x+1}},
\]
and the trace condition implies that $y/x$ and $(y+1)/(x+1)$ are both roots of unity. Similarly $z/x$ and $(z+1)/(x+1)$ are roots of unity. But this means in particular that $|x|=|y|=|z|$ and $|x+1|=|y+1|=|z+1|$, which is impossible for three distinct complex numbers.
\end{pf}

\begin{propn}\label{int1}
Suppose $M$ is a finite \prm on $\bbc^2$, with $\card{\imm M}=\card{\kerm M}=3$ and $\card{\imm M\cap \kerm M}=1$. Then $M$ is isomorphic to $\bw w$ for some root of unity $w\neq\pm1$.
\end{propn}

\begin{pf}
Without loss of generality we assume that $\spn11,\spn01\in\kerm M$, while $\spn11,\spn10\in\imm M$. Write the remaining kernel as $\spn{1-w}1$ and the remaining image as $\spn1{1-x}$, where $w,x\neq0,1$. Then $M$ contains the projections
\[
\fmx1010,\ \fmx10{1-x}0,\ \fmx1{-1}00,\ \fmx1{w-1}00,
\]
and the trace condition shows that $w$, $x$ and $w+x-wx$ are roots of unity. The only way this can happen is if $x=-w$, which means in particular that $w\neq-1$, so $M=\bw w$.
\end{pf}

Now we consider the case where $\card{\imm M\cap \kerm M}\gs2$. We fix a primitive cube root of unity $\omega\in\bbc$.

\begin{lemma}\label{int2}
Suppose $M$ is a finite \prm on $\bbc^2$, and that $\spn10$ and $\spn01$ are both images and kernels of $M$.
\begin{enumerate}
\item
Suppose $\spn{-1}1$ is a kernel of $M$. Then the only possible other images of $M$ are $\spn{-1}1$, $\spn\omega1$ and $\spn{\omega^2}1$.
\item
Suppose $\spn{-1}1$ is an image of $M$. Then the only possible other kernels of $M$ are $\spn{-1}1$, $\spn\omega1$ and $\spn{\omega^2}1$.
\item
Suppose $\spn\omega1$ is a kernel of $M$. Then the only possible other images of $M$ are $\spn{-1}1$, $\spn\omega1$ and $\spn{-\omega^2}1$.
\item
Suppose $\spn{\omega^2}1$ is a kernel of $M$. Then the only possible other images of $M$ are $\spn{-1}1$, $\spn{\omega^2}1$ and $\spn{-\omega}1$.
\end{enumerate}
\end{lemma}

\begin{pf}
For (1), suppose $\spn x1$ is an image, with $x\neq0,-1$. Then $M$ contains the projections
\[
\fmx0011,\ \fmx0x01,\ \fmx1100,\ \fmx10{x^{-1}}0
\]
and the trace condition implies that both $x+1$ and $(x+1)/x$ are roots of unity. The only way this can happen is if $x=\omega^{\pm1}$.

(2) now follows by duality, and (3) and (4) are proved in a similar way.
\end{pf}

Now we are ready to prove our main result.

\begin{pf}[Proof of \cref{classify}]
If $\card{\imm M}$ or $\card{\kerm M}$ equals $2$, then by \cref{onlytwo} either $M$ or its dual is equivalent to $\as Si$ for some $S,i$. So assume that $\card{\imm M}\gs3$ and $\card{\kerm M}\gs3$.

If $\card{\imm M\cup \kerm M}=3$, then $M$ is equivalent to $\mna2$, so assume that $\card{\imm M\cup \kerm M}\gs4$. \cref{int0} shows that $\imm M\cap \kerm M\neq\emptyset$; in fact, if $\calk\subseteq\kerm M$ and $\call\subseteq\imm M$ with $\card\calk=\card\call=3$, then $\calk\cap\call\neq\emptyset$. \cref{int1} shows that if $\card{\imm M\cap \kerm M}=1$ (forcing $\card{\imm M}=\card{\kerm M}=3$) then $M$ is isomorphic to $\bw w$ for some $w$. So assume $\card{\imm M\cap \kerm M}\gs2$.  Then without loss of generality we can assume that $\spn10$ and $\spn01$ are both images and kernels of $M$, and that $\spn{-1}1$ is also an image. If there are no other images, then \cref{int2}(2) shows that (up to equivalence) $M$ is one of $\ci0,\ci1,\ci2,\ci3$. So we can assume that there are at least four images, and dually we can assume that there are at least four kernels. So at least two of $\spn{-1}1$, $\spn\omega1$ and $\spn{\omega^2}1$ are kernels. But if $\spn\omega1$ and $\spn{\omega^2}1$ are kernels, then \cref{int2}(3,4) show that there can only be three images, a contradiction. So (\wolog) the images of $M$ are $\spn10$, $\spn01$, $\spn{-1}1$, $\spn\omega1$. Now \cref{int2}(2,3) show that the kernels of $M$ are also $\spn10$, $\spn01$, $\spn{-1}1$, $\spn\omega1$, so that $M=\ci4$.
\end{pf}

\begin{rmks}
\begin{enumerate}[beginthm]
\item
From \cref{classify} we can deduce a classification of \emph{minimal} irreducible \prms on $\bbc^2$; that is, those with no irreducible projection submonoid. First we observe that an irreducible \prm $M$ is minimal \iff it has exactly two images and two kernels. The ``if'' part follows from \cref{irrconds}. For the ``only if'', note that a \cmp \prm with exactly two kernels and two images is irreducible except in the case where the two kernels are the same as the two images. If $M$ is irreducible and has at least three images or at least three kernels, then we can easily choose two kernels $k,l$ and two images $s,t$ such that $\{k,l\}\neq\{s,t\}$. Then the \prm with kernels $k,l$ and images $s,t$ is irreducible, so $M$ is not minimal.

We can read off from \cref{onlytwo} the irreducible \prms with exactly two images and two kernels; they are the monoids $\as S0$ with $\card S=2$, and $\as{\{1\}}1$.
\item
We can also easily read off the self-dual irreducible \prms: these are
\[
\as S0\ (\text{for }\card S=2),\quad \as{\{1\}}1,\quad \bw{\mathrm i},\quad\ci0,\ \ci1,\ \ci5.
\]

\item
From the classification over $\bbc$ we can immediately deduce a classification over $\bbr$: if $M$ is a finite irreducible \prm on $\bbr^2$, then $M$ or its dual is equivalent to one of the monoids
\[
\as{\{\pm1\}}0,\ \as{\{1\}}1,\ \as{\{\pm1\}}1,\ \as{\{1\}}2,\ \as{\{\pm1\}}2,\ \ci0,
\]
of orders $9$, $6$, $14$, $8$, $18$, $20$ respectively. The first five of these are all submonoids of $\mnb2^2=\as{\{\pm1\}}2$, and the last is just $\mna2$.
\end{enumerate}
\end{rmks}

\section{Finite irreducible \prms on $\bbr^3$}\label{r3classsec}

In this section we consider finite irreducible \prms on $\bbr^3$. In \cref{classsec} we saw that a finite irreducible \prm on $\bbr^2$ is (up to duality) equivalent to a submonoid of $\mna2$ or $\mnb2^2$. Here we prove the corresponding statement for $\bbr^3$.

\begin{thm}\label{r3classn}
Suppose $M$ is a finite irreducible \prm on $\bbr^3$. Then $M$ or its dual is equivalent to a submonoid of $\mna3$ or $\mnb3^2$.
\end{thm}

Throughout this section we assume that $M$ is a finite \cmp \prm on $\bbr^3$, the sum of the kernels of $M$ is $\bbr^3$, and the intersection of the images of $M$ is $0$. 
We say that $M$ is \emph{split} if there are $K\in\kerm M$ and $A\in\imm M$ such that all the images of $M$ other than $A$ contain $K$, and all the kernels of $M$ other than $K$ are contained in $A$.

We will view elements of $M$ as $3\times3$ matrices acting on $\bbr^3$ on the left. Given $a,b,c\in\bbr$, we write $\spa abc$ for the span of the vector $\left(\begin{smallmatrix}a\\b\\c\end{smallmatrix}\right)$.

To begin with, we consider the case where $M$ satisfies the following additional condition:
\[
\text{if $K\in\kerm M$ and $A,B,C\in\imm M$ with $A\cap B\cap C=0$, then $K\subset A\cup B\cup C$.}\tag*{\starr}
\]

\begin{lemma}\label{star3}
Suppose $M$ satisfies \starr and $\card{\imm M}=3$. Then either $M$ is split, or 
$M$ is equivalent to a submonoid of $\mnb3^2$.
\end{lemma}


\begin{pf}
By replacing $M$ with an equivalent \prm, we can assume that the images of $M$ are the three coordinate planes. Condition \starr says that every kernel of $M$ is contained in at least one of these planes.

Let's assume that $M$ is not split. Then at least two of the following statements must be true:
\begin{itemize}
\item
there is $a\neq0$ such that $\spa1a0\in\kerm M$;
\item
there is $b\neq0$ such that $\spa10b\in\kerm M$;
\item
there is $c\neq0$ such that $\spa01c\in\kerm M$.
\end{itemize}
By permuting the coordinates if necessary, we assume that the first two statements are true. By rescaling the coordinates if necessary, we can assume that $a=b=1$, so that $\spa110$ and $\spa101$ are kernels of $M$. Now we claim that the only possible kernels of $M$ are
\[
\spa100,\ \spa010,\ \spa001,\ \spa110,\ \spa1{-1}0,\ \spa101,\ \spa10{-1},\ \spa011,\ \spa01{-1},
\]
which is the same as saying that $M$ is a submonoid of $\mnb3^2$.

Suppose first that $\spa1d0$ is a kernel, with $d\neq0,1$. Then $M$ contains the projections
\[
\begin{pmatrix}1&-1&0\\0&0&0\\0&0&1\end{pmatrix},\ \begin{pmatrix}0&0&0\\-d&1&0\\0&0&1\end{pmatrix}
\]
and so contains
\[
\begin{pmatrix}1&-1&0\\0&0&0\\0&0&1\end{pmatrix}\begin{pmatrix}0&0&0\\-d&1&0\\0&0&1\end{pmatrix}=\begin{pmatrix}d&-1&0\\0&0&0\\0&0&1\end{pmatrix},
\]
and hence contains all powers of this matrix. For this matrix to have only finitely many different powers, $d$ must equal $-1$. Similarly, if $\spa10d$ is a kernel of $M$ with $d\neq0,1$, then $d=-1$.

Now suppose $\spa01d$ is a kernel of $M$ with $d\neq0$. Then $M$ contains the projections
\[
\begin{pmatrix}1&0&-1\\0&1&0\\0&0&0\end{pmatrix},\ \begin{pmatrix}1&0&0\\0&0&0\\0&-d&1\end{pmatrix},\ \begin{pmatrix}0&0&0\\-1&1&0\\0&0&1\end{pmatrix},
\]
and so contains
\[
\begin{pmatrix}1&0&-1\\0&1&0\\0&0&0\end{pmatrix}\begin{pmatrix}1&0&0\\0&0&0\\0&-d&1\end{pmatrix}\begin{pmatrix}0&0&0\\-1&1&0\\0&0&1\end{pmatrix}=\begin{pmatrix}-d&d&-1\\0&0&0\\0&0&0\end{pmatrix}
\]
In order for this matrix to have only finitely many powers, $d$ must equal $\pm1$. This proves our claim, and hence proves the \lcnamecref{star3}.
\end{pf}

\begin{lemma}\label{star4}
Suppose $M$ satisfies \starr, and $\card{\imm M}=4$. Then either $M$ is split, or $M$ is equivalent to a submonoid of $\mna3$.
\end{lemma}

\begin{pf}
We consider two possibilities for the configuration of the images of $M$:
\begin{enumerate}[(a),ref=(\alph*)]
\item\label{special4}
there are three of the images which meet in a line.
\item\label{general4}
any three of the images of $M$ have zero intersection;
\end{enumerate}
First suppose \ref{special4} holds. Write the images of $M$ as $A,B,C,D$, with $B\cap C\cap D$ being a line. Suppose $K$ is a kernel of $M$ which is not contained in $A$. Then \starr applied to $A,B,C$ shows that $K\subset B\cup C$. Symmetrically, $K$ is contained in $B\cup D$ and in $C\cup D$, so that $K=B\cap C\cap D$. Therefore $M$ is split.

Now suppose \ref{general4} holds. Then \starr implies that every kernel of $M$ lies in two of the images (since otherwise there is a kernel $K$ and three images with zero intersection none of which contains $K$, contradicting \starr). So the only possible kernels of $M$ are the intersections of pairs of images. But given any four planes satisfying \ref{general4}, the \cmp \prm with these four planes as its images and their pairwise intersections as its kernels is equivalent to $\mna3$. So $M$ is equivalent to a submonoid of $\mna3$.
\end{pf}

\begin{lemma}\label{star5}
Suppose $M$ satisfies \starr and $\card{\imm M}\gs5$. Then $M$ is split.
\end{lemma}

\begin{pf}
Let $n=\card{\imm M}$. We prove the \lcnamecref{star5} by induction on $n$, starting with the case $n=5$. If we write the images of $M$ as $A,B,C,D,E$, then (permuting these labels if necessary) exactly one of the following holds:
\begin{enumerate}[(a),ref=(\alph*)]
\item\label{general5}
any three of $A,B,C,D,E$ have zero intersection;
\item
$A\cap B\cap C$ is a line, but the intersection of any other set of three images of $M$ is zero;
\item
$A\cap B\cap C$ and $C\cap D\cap E$ are lines, but the intersection of any other set of three images of $M$ is zero;
\item
$B\cap C\cap D\cap E$ is a line, but the intersection of $A$ and any two other images is zero.
\end{enumerate}
We address these cases individually.
\begin{enumerate}[(a)]
\item
In this case \starr gives a contradiction, since if $K\in\kerm M$ then $K$ is contained in at most two of the images $A,B,C,D,E$; so there will be three images (with zero intersection) none of which contains $K$, contradicting \starr.
\item
In this case \starr shows that every kernel $K$ is contained in two of the planes $A,B,D,E$, and also in two of $A,C,D,E$, and also in two of $B,C,D,E$. Hence the line $D\cap E$ is the only possible kernel of $M$, a contradiction.
\item
In this case \starr shows that the only possible kernels of $M$ are the pairwise intersections of the lines $A,B,D,E$. But \starr applied to $A,C,D$ shows that the line $B\cap E$ cannot be a kernel of $M$. Similarly none of $B\cap D$, $A\cap E$, $A\cap D$ are kernels, so that $M$ has at most two kernels, a contradiction.
\item
In this case we argue as in the proof of \cref{star4} to show that the only possible kernel of $M$ which is not contained in $A$ is the line $B\cap C\cap D\cap E$, so that $M$ is split.
\end{enumerate}
Now we suppose $n\gs6$ and assume that the \lcnamecref{star5} holds with $n$ replaced by $n-1$. If we can show that there are $n-1$ images of $M$ that meet in a line, then we can argue as in the proof of \cref{star4} to show that $M$ is split.

Given $n-1$ of the images of $M$ which do not meet in a line, let $M^-$ be the \cmp \prm with these images and with the same kernels as $M$. Then the inductive hypothesis shows that $M^-$ is split, so that $n-2$ of these images meet in a line. So the following statement is true:
\[
\text{given any $n-1$ of the images of $M$, at least $n-2$ of them meet in a line.}\tag*{($\dagger$)}
\]
Write the images of $M$ as $A_1,\dots,A_n$. If $A_1\cap\dots\cap A_{n-1}$ is a line, then we are done, so assume otherwise. By ($\dagger$) some $n-2$ of $A_1,\dots,A_{n-1}$ meet in a line; by reordering if necessary, we assume that $L=A_1\cap\dots\cap A_{n-2}$ is a line, and that $A_{n-1}$ does not contain $L$. Now we just need to show that $A_n$ contains $L$. Applying ($\dagger$) to $A_2,\dots,A_n$, we find that some $n-2$ of these must meet in a line. These $n-2$ planes include at least two of $A_2,\dots,A_{n-2}$ because $n\gs6$, so this line must be $L$. Therefore all but one of the lines $A_2,\dots,A_n$ contain $L$; by assumption $A_{n-1}$ does not, and therefore $A_n$ contains $L$, as required.
\end{pf}

Now we prove some preliminary results to deal with the case where neither $M$ nor its dual satisfies \starr.

\begin{lemma}\label{r32}
Suppose the three coordinate planes are images of $M$, and $\spa111$ and $\spa1ab$ are kernels, for $a,b\in\bbr$ not both equal to $1$. Then at least one of $a$ and $b$ equals $0$, and the other equals $0$ or $-1$.
\end{lemma}

\begin{pf}
By assumption $M$ contains the projections
\[
\begin{pmatrix}0&0&0\\-a&1&0\\-b&0&1\end{pmatrix},\ \begin{pmatrix}1&-1&0\\0&0&0\\0&-1&1\end{pmatrix},
\]
and therefore contains the matrix
\[
p=\begin{pmatrix}0&0&0\\-a&1&0\\-b&0&1\end{pmatrix}\begin{pmatrix}1&-1&0\\0&0&0\\0&-1&1\end{pmatrix}=\begin{pmatrix}0&0&0\\-a&a&0\\-b&b-1&1\end{pmatrix}.
\]
In order for $p$ to have only finitely many different powers, we need $a\in\{0,1,-1\}$. Symmetrically, $b\in\{0,1,-1\}$.

Now assume that neither $a$ nor $b$ equals $0$. Then (by permuting the coordinate axes if necessary) we can assume that $a=1$ and $b=-1$. Now the matrix $p$ above becomes
\[
\begin{pmatrix}0&0&0\\-1&1&0\\1&-2&1\end{pmatrix},
\]
and a simple induction shows that
\[
p^n=\begin{pmatrix}0&0&0\\-1&1&0\\2n-1&-2n&1\end{pmatrix}
\]
for any $n\in\bbn$, so that $M$ is infinite, a contradiction.

So at least one of $a$ and $b$ is zero. Now (up to symmetry) it only remains to rule out the case where $a=1$ and $b=0$. In this case
\[
p=\begin{pmatrix}0&0&0\\-1&1&0\\0&-1&1\end{pmatrix},
\]
and therefore
\[
p^n=\begin{pmatrix}0&0&0\\-1&1&0\\n-1&-n&1\end{pmatrix}
\]
for any $n$, and again $M$ is infinite.
\end{pf}

%

\begin{lemma}\label{next3im}
Suppose the three coordinate planes are images of $M$. Then $\spa111$, $\spa1{-1}0$ and $\spa10{-1}$ cannot all be kernels of $M$.
\end{lemma}

\begin{pf}
Suppose $M$ does contain all three of the given kernels. Then $M$ contains the projections
\[
\begin{pmatrix}0& 0& 0\\0& 1& 0\\1& 0& 1\end{pmatrix},\ \begin{pmatrix}1& 1& 0\\0& 0& 0\\0& 0& 1\end{pmatrix},\ \begin{pmatrix}1& 0& -1\\0& 1& -1\\0& 0& 0\end{pmatrix},
\]
and therefore contains
\[
\begin{pmatrix}0& 0& 0\\0& 1& 0\\1& 0& 1\end{pmatrix}\begin{pmatrix}1& 1& 0\\0& 0& 0\\0& 0& 1\end{pmatrix}\begin{pmatrix}1& 0& -1\\0& 1& -1\\0& 0& 0\end{pmatrix}=\begin{pmatrix}0&0&0\\0&0&0\\1&1&-2\end{pmatrix}.
\]
Obviously all powers of this matrix are distinct, so that $M$ is infinite.
\end{pf}

\begin{lemma}\label{nextnext3im}
Suppose the three coordinate planes are images of $M$. Then $\spa111$, $\spa1{-1}0$ and $\spa100$ cannot all be kernels of $M$.
\end{lemma}

\begin{pf}
Suppose $M$ does contain all three of the given kernels.  Then $M$ contains the projections
\[
\begin{pmatrix}0&0&0\\0&1&0\\0&0&1\end{pmatrix},\ \begin{pmatrix}1&-1&0\\0&0&0\\0&-1&1\end{pmatrix},\ \begin{pmatrix}0&0&0\\1&1&0\\0&0&1\end{pmatrix},\ \begin{pmatrix}1&0&-1\\0&1&-1\\0&0&0\end{pmatrix},
\]
and therefore contains
\[
\begin{pmatrix}0&0&0\\0&1&0\\0&0&1\end{pmatrix}\begin{pmatrix}1&-1&0\\0&0&0\\0&-1&1\end{pmatrix}\begin{pmatrix}0&0&0\\1&1&0\\0&0&1\end{pmatrix}\begin{pmatrix}1&0&-1\\0&1&-1\\0&0&0\end{pmatrix}=\begin{pmatrix}0&0&0\\0&0&0\\-1&-1&2\end{pmatrix}.
\]
Obviously all powers of this matrix are distinct, so that $M$ is infinite.
\end{pf}

\begin{lemma}\label{notirred}
Suppose $A,B,C$ are images of $M$ with $A\cap B\cap C=0$, and $K$ is a kernel of $M$ with $K\not\subseteq A\cup B\cup C$. Then
\[
\kerm M\subseteq\{K,A\cap B,A\cap C,B\cap C\}.
\]
\end{lemma}

\begin{pf}
By replacing $M$ with an equivalent \prm, we can assume that $A,B,C$ are the three coordinate planes, and $K=\spa111$. By \cref{r32} (and by permuting coordinates) the only possible kernels other than $\spa111$ and the three coordinate axes are $\spa1{-1}0$, $\spa10{-1}$ and $\spa01{-1}$. Suppose one of these is a kernel, say $\spa1{-1}0$. Then by \cref{next3im,nextnext3im} neither $\spa10{-1}$ nor $\spa100$ is a kernel, and by permuting coordinates neither is $\spa01{-1}$ or $\spa010$. But this means that the only possible kernels are $\spa111$, $\spa1{-1}0$ and $\spa001$, so that the kernels of $M$ do not span $\bbr^3$, contrary to assumption.
\end{pf}

\begin{pf}[Proof of \cref{r3classn}]
Suppose $M$ satisfies \starr. Then by \cref{star3,star4,star5} $M$ is equivalent to a submonoid of $\mna3$ or $\mnb3^2$, or $M$ is split. But if $M$ is split, then $M$ is reducible, contrary to assumption.

So we can assume that $M$ does not satisfy \starr. By applying the above paragraph to the dual \prm $M^\ast$, we can assume that $M^\ast$ does not satisfy \starr either. In this case we will derive a contradiction.

The assumption that $M$ does not satisfy \starr means that we can find $K,A,B,C$ as in \cref{notirred}. From that \lcnamecref{notirred} we deduce that $\card{\kerm M}\ls4$. By duality, we can also deduce that $\card{\imm M}\ls4$. Now suppose $A,B,C,K$ are as in \cref{notirred}. If $A,B,C$ are the only images of $M$, then the dual of $M$ satisfies \starr: given any three kernels of $M$, each of $A,B,C$ contains at least one of them. So our assumption that $M^\ast$ does not satisfy \starr means that $\card{\imm M}=4$. Dually, we deduce that $\card{\kerm M}=4$, so that $\kerm M=\{K,A\cap B,A\cap C,B\cap C\}$. This implies in particular that no three kernels of $M$ are coplanar. Dually, no three images meet in a line, so the remaining image $D\neq A,B,C$ cannot contain any of the lines $A\cap B$, $A\cap C$, $B\cap C$.

Now we claim that $D$ contains $K$. If it does not, then $A\cap B\cap D=0$ and $K\not\subseteq A\cup B\cup D$, but $M$ has kernels other than $K,A\cap B,A\cap D,B\cap D$ (namely $A\cap C$), so $A,B,D$ violate \cref{notirred}. 

So $D$ must contain $K$. Now for definiteness let us assume that $A$, $B$ and $C$ are the three coordinate planes and $K=\spa111$. Then we can write $D$ as the orthogonal complement of the line $\spa a{-1}{1-a}$ for some $a$. The assumption that $D$ contains none of the coordinate axes means that $a\neq0,1$.

Now $M$ contains the projections
\[
\begin{pmatrix}1&0&0\\0&0&0\\0&0&1\end{pmatrix},\ \begin{pmatrix}0&0&0\\-1&1&0\\-1&0&1\end{pmatrix},\ \begin{pmatrix}0&a^{-1}&1-a^{-1}\\0&1&0\\0&0&1\end{pmatrix},
\]
amd so contains
\[
\begin{pmatrix}1&0&0\\0&0&0\\0&0&1\end{pmatrix}\begin{pmatrix}0&0&0\\-1&1&0\\-1&0&1\end{pmatrix}\begin{pmatrix}0&a^{-1}&1-a^{-1}\\0&1&0\\0&0&1\end{pmatrix}=\begin{pmatrix}0&0&0\\0&0&0\\0&-a^{-1}&a^{-1}\end{pmatrix}.
\]
This last matrix must have only finitely many different powers, forcing $a=-1$.

But now $M$ contains the projections
\[
\begin{pmatrix}1&0&0\\0&1&0\\0&0&0\end{pmatrix},\ \begin{pmatrix}1&-1&0\\0&0&0\\0&-1&1\end{pmatrix},\ \begin{pmatrix}1&0&0\\-1&0&2\\0&0&1\end{pmatrix},
\]
and so contains
\[
\begin{pmatrix}1&0&0\\0&1&0\\0&0&0\end{pmatrix}\begin{pmatrix}1&-1&0\\0&0&0\\0&-1&1\end{pmatrix}\begin{pmatrix}1&0&0\\-1&0&2\\0&0&1\end{pmatrix}=\begin{pmatrix}2&0&-2\\0&0&0\\0&0&0\end{pmatrix},
\]
and we see that $M$ is infinite, a contradiction.
\end{pf}


\section{Affine \prms}

\subsection{Definitions}

In this final section we broaden our definitions, and consider monoids generated by \emph{affine projections}: that is, idempotent affine-linear maps whose image is an affine subspace of $V$ of codimension $1$. The \prms considered in preceding sections (whose elements are all linear maps) will be called \emph{linear} \prms in this section. Affine \prms are analogous to affine reflection groups, which play an important role in Lie theory. However, while properly affine reflection groups (i.e.\ those which are not equivalent to linear reflection groups) are necessarily infinite in characteristic $0$, properly affine \prms can be finite. In this section we will mainly restrict attention to definitions and examples, leaving a more comprehensive exploration of affine \prms to a future paper.

We start by setting out some definitions and basic theory. If $p:V\to V$ is an affine projection, then its image $\im(p)$ is an affine subspace of $V$. We abuse notation and define the \emph{kernel} of $p$ to be $\ker(p)=\lset{p(v)-v}{v\in V}$. Thus $\ker(p)$ is a (linear) subspace of $V$, of dimension $1$ (but is not the kernel in the usual sense of being the preimage of the zero vector under~$p$).

If $m:V\to V$ is an affine-linear map, then we can define a linear map $\ud m:V\to V$ by $\ud m(v)=m(v)-m(0)$. Clearly if $p$ is an affine projection then $\ud p$ is a linear projection: the image $\im(\ud p)$ is the linear subspace obtained as a translation of $\im(p)$, and the kernel $\ker(\ud p)$ equals $\ker(p)$.

For any affine-linear maps $l,m$, it is clear that $\ud{lm}=\ud l\ud m$. Thus if $M$ is an affine \prm, then the monoid $\ud M=\lset{\ud m}{m\in M}$ is a linear \prm, naturally isomorphic to a quotient of $M$. We call $\ud M$ the \emph{linear \prm underlying $M$}.

The notion of \cmpness introduced in \cref{irrsec} naturally extends to affine \prms: $M$ is \cmp if whenever $K$ is a kernel of $M$ and $L$ is an image of $M$ such that $K+L=V$ (or equivalently, $K$ is not contained in any translate of $L$) the affine projection $\pwi KL$ lies in $M$. It is easy to see that if $M$ is \cmp, then so is $\ud M$; a partial converse is given by following.

\begin{propn}\label{affcomp}
Suppose $M$ is an affine \prm, and that the underlying linear \prm $\ud M$ is \cmp. Suppose that one of the following holds.
\begin{enumerate}
\item\label{nopara}
$M$ does not have two parallel images.
\item\label{fini}
$\nchar(\bbf)=0$ and $M$ is finite.
\end{enumerate}
Then $M$ is \cmp.
\end{propn}

\begin{pf}
First suppose (\ref{nopara}) holds. Suppose $K$ is a kernel of $M$ and $L$ an image of $M$ with $K+L=V$, and let $\ud L$ be the linear subspace obtained by translating $L$. Then $\pwi K{\ud L}\in \ud M$, because $\ud M$ is \cmp, which means that $M$ must contain an affine projection with kernel $K$ and image parallel to $L$. But (\ref{nopara}) tells us that the only image parallel to $L$ is $L$ itself, so $\pwi KL$ lies in $M$.

Now suppose (\ref{fini}) holds. First we note that if $p,q$ are projections in $M$ with distinct but parallel images, then $\ker(p)=\ker(q)$. To see this, suppose that the kernels are distinct, and consider the effect of $pq$ on $\im(p)$. Because $\im(p)$ and $\im(q)$ are parallel and distinct, there is a non-zero vector $w\in\ker(q)$ such that $q(v)=v+w$ for all $v\in\im(p)$. Similarly, there is a non-zero vector $x\in\ker(p)$ such that $p(v)=v+x$ for all $x\in\im(q)$. So $pq$ acts on $\im(p)$ as translation by $w+x$. The assumption that $\ker(p)$ and $\ker(q)$ are not parallel means that $w$ and $x$ are linearly independent, so that $w+x\neq0$. So $pq$ acts on $\im(p)$ by a non-zero translation, so has infinite order, contradicting (2).

Now suppose that $K$ is a kernel of $M$ and $L$ an image of $M$ with $K+L=V$. Then, as in the first part of the proof, $M$ must contain $\pwi K{L'}$ for some $L'$ parallel to $L$. If $L'=L$, then we have the desired projection $\pwi KL$ in $M$. So suppose $L'\neq L$. Let $p$ be a projection in $M$ with image $L$. Then from the paragraph above, $\ker(p)=K$; so $p=\pwi KL\in M$.\qedhere
\end{pf}

We extend the notion of irreducibility to affine \prms: $M$ is irreducible if there is no proper affine subspace of dimension greater than $0$ which is mapped to itself by every element of $M$. (Note that if there is an invariant affine subspace of dimension $0$, i.e.\ a point fixed by every element of $M$, then $M$ is equivalent to a linear \prm.)

It is clear that if the underlying linear \prm $\ud M$ is irreducible, then so is $M$. However, the converse need not hold. For example, take $V$ of dimension $2$, and let $K$ and $L$ be any two (linear) lines in $V$. Choose affine lines $L_1,\dots,L_r$ (with $r\gs2$) parallel to $L$, and let $J_1,\dots,J_s$ be any affine lines not parallel to $L$. Let $M$ be the affine \prm generated by
\begin{itemize}
\item
the projections $\pwi K{L_i}$ for $i=1,\dots,r$, and
\item
the projections $\pwi L{J_i}$ for $i=1,\dots,s$.
\end{itemize}
Then it is easy to check that $M$ is finite, and it is irreducible (unless $s=1$ and $J_1$ is parallel to $K$). But the underlying \prm $\ud M$ is reducible, because it has $L$ as an invariant subspace. Note also that (provided at least one of $J_1,\dots,J_s$ is not parallel to $K$) $M$ is not \cmp, so we do not have an analogue of \cref{irrconds} for affine \prms.

\subsection{Examples}

We finish by giving two examples of affine \prms which derive from \cref{typeasec,typebsec}.

For our first example, recall the construction in \cref{typeasec}: we take $V$ to be a vector space of dimension $n$, and that we embed $V$ as a vector space $V^+$ of dimension $n+1$, with basis $\{e_0,\dots,e_n\}$ chosen so that $V=\lset{\sum_i\la_ke_k}{\sum_k\la_k=0}$. For $i\neq j$ we defined the projection $p^+_{ij}:V^+\to V^+$ to have image $\lspan{e_k}{k\neq i}$ and kernel $\lan e_i-e_j\ran$, and we defined $\mnap n$ to be the \prm on $V^+$ generated by the projections $p_{ij}^+$. Then $\mna n$ is the \prm on $V$ generated by the restrictions to $V$ of the projections $p^+_{ij}$.

We saw in \cref{typeasec} that the semigroup of non-invertible elements of $\mnap n$ is isomorphic to the semigroup of non-invertible functions from $\{0,\dots,n\}$ to itself; this isomorphism can be realised directly by observing that the elements of $\mnap n$ map the basis $B=\{e_0,\dots,e_n\}$ to itself, with each non-bijective map from $B$ to $B$ corresponding to a unique element of $\mnap n$.

Now define $U$ to be the affine subspace
\[
\lset{\sum_k\la_ke_k}{\sum_k\la_k=1}
\]
of $V^+$. For $i\neq j$, define $q_{ij}:U\to U$ to be the restriction of $p^+_{ij}$ to $U$. Then $q_{ij}$ is an affine projection: it has kernel $\lan e_i-e_j\ran$ and image $\lset{\sum_k\la_ke_k}{\la_i=0,\ \sum_k\la_k=1}$. Define $\mnc n$ to be the affine \prm on $U$ generated by the projections $q_{ij}$. Then $\mnc n$ is naturally isomorphic to a quotient of $\mnap n$. In fact $\mnc n$ is isomorphic to $\mnap n$, because $U$ contains the basis $B$, so each non-bijective map $B\to B$ still corresponds to a unique element of $\mnc n$.

\medskip
Now we give an example based on the dual of the monoid $\mnb n^t$ constructed in \cref{typebsec}. As in that section, suppose $Z$ is a finite subgroup of $\bbf^\times$, let $t=\card Z$, and take $V$ to be $n$-dimensional over $\bbf$ with basis $\{e_1,\dots,e_n\}$. Now let $X$ be a subset of $\bbf$. Consider the following two kinds of projections:
\begin{itemize}
\item
for each $i\neq j$ and each $z\in Z$, take the projection with kernel $\lan e_i\ran$ and image $\lset{\sum_k\la_ke_k}{\la_i=z\la_j}$.
\item
for each $i$ and each $x\in X$, take the projection with kernel $\lan e_i\ran$ and image $\lset{\sum_k\la_ke_k}{\la_i=x}$.
\end{itemize}

Let $\mnd n^t(X)$ be the affine \prm generated by these projections. Then it is easy to see that the linear \prm underlying $\mnd n^t(X)$ is the dual of $\mnb n^t$.

\begin{propn}\label{affbfin}
If $X$ is finite, then $\mnb n^t(X)$ is finite.
\end{propn}

\begin{pf}
We write the elements of $V$ as column vectors with respect to the basis $\{e_1,\dots,e_n\}$. Then an affine-linear map on $V$ can be written in the form $v\mapsto Av+b$, where $A$ is an $n\times n$ matrix and $b\in v$. Now let $N^t(X)$ denote the set of pairs $(A,b)$ for which:
\begin{itemize}
\item
each row of $A$ has at most one non-zero entry, and all non-zero entries of $A$ lie in $Z$;
\item
for each $k$, if row $k$ of $A$ contains no non-zero entries, then $b_k\in\lset{zx}{z\in Z,\ x\in X}\cup\{0\}$, while if row $k$ of $A$ contains a non-zero entry, then $b_k=0$.
\end{itemize}
We claim that for every $m\in M$ the corresponding pair $(A,b)$ lies in $N^t(X)$. Since $N^t(X)$ is obviously finite, this will prove the result.

First observe that for each of the generating projections $p\in M$, the corresponding pair $(A,b)$ lies in $N^t(X)$:
\begin{itemize}
\item
if $p$ has kernel $\lan e_i\ran$ and image $\lset{\sum_k\la_ke_k}{\la_i=z\la_j}$, then $b=0$ and $A$ agrees with the identity matrix except that its $(i,i)$-entry is $0$ and its $(i,j)$-entry is $z$;
\item
if $p$ has kernel $\lan e_i\ran$ and image $\lset{\sum_k\la_ke_k}{\la_i=x}$, then $A$ agrees with the identity matrix except that its $(i,i)$-entry is zero, while $b_i=x$ and $b_k=0$ for all $k\neq i$.
\end{itemize}
So to prove the \lcnamecref{affbfin} we just need to show that if two affine-linear maps correspond to pairs $(A,b)$ lying in $N^t(X)$, then so does their product. This amounts to showing that if $(A,b),(C,d)\in N^t(X)$, then $(AC,Ad+b)\in N^t(X)$, which is a simple exercise.
\end{pf}

By analysing the examples in \cref{someexsec}, one can show that every finite irreducible affine \prm on $\bbc^2$ is equivalent either to a linear \prm or to a submonoid of one of these examples. It is clear that a great deal more work is needed to understand affine \prms in general. It may be that all finite affine \prms arise as straightforward variations of linear \prms, as with the examples above, or there may be some hidden richness. We leave this investigation for future work.


\begin{thebibliography}{99}
\setlength\itemsep{0pt}

\backrefparscanfalse

\bibi{cox1}{C1}{H.~S.~M.~Coxeter}{Discrete groups generated by reflections}{Ann.\ of Math.}{35}{1934}{588--621}

\bibi{cox2}{C2}{H.~S.~M.~Coxeter}{The complete enumeration of finite groups of the form $r_i^2=(r_ir_j)^{k_{ij}}=1$}{J.~ London Math.\ Soc.}{1935}{10}{21--25}

\bibi{howie}{H1}{J.\ M.\ Howie}{The subsemigroup generated by the idempotents of a full transformation semigroup}{J.~London Math.\ Soc.}{41}{1966}{707--716}

\bibi{howie2}{H2}{J.\ M.\ Howie}{Idempotent generators in finite full transformation semigroups}{Proc.\ Roy.\ Soc.\ Edinburgh Sect.\ A}{81}{1978}{317--323}

\bibi{st}{ST}{G.\ C.\ Shephard \& J.\ A.\ Todd}{Finite unitary reflection groups}{Canadian J.\ Math.}{6}{1954}{274--304}

\bibi{zs1}{ZS1}{A.~Zalesskii \& A.~ Serezhkin}{Linear groups generated by transvections}{Izv.\ Akad.\ Nauk SSSR Ser.\ Mat.}{40}{1976}{26--49}

\bibitem[ZS2]{zs2}A.~Zalesskii \& A.~ Serezhkin, `Linear groups that are generated by pseudoreflections', \textit{Vesc\=\i{} Akad.\ Navuk BSSR Ser.\ F\=\i z.-Mat.\ Navuk} \textbf{1977}, 9--16.\backrefprint\renewcommand\con{}\renewcommand\cons{}

\bibi{zs3}{ZS3}{A.~Zalesskii \& A.~ Serezhkin}{Finite linear groups generated by reflections}{Izv.\ Akad.\ Nauk SSSR Ser.\ Mat.}{44}{1980}{1279--1307}
\end{thebibliography}
\end{document}